\documentclass[11pt,a4paper]{amsart}

\usepackage{amsfonts, amssymb, amsmath, amsthm, latexsym, enumerate,
  array, hyperref, cite, mdwlist, mathrsfs}
\usepackage[all]{xy}
\usepackage{stmaryrd}
\usepackage[bitstream-charter]{mathdesign}

\newcommand{\mathbold}{\mathbf}

\newtheorem{thm}{Theorem}[section]
\newtheorem{lem}[thm]{Lemma}

\newtheorem{prop}[thm]{Proposition}
\newtheorem{THM}{Theorem}

\theoremstyle{definition}
\newtheorem{rmk}[thm]{Remark}
\newtheorem{dfn}{Definition}

\newtheorem{qst}{Question}
\newtheorem{step}{Step}

\newcommand{\OO}{\mathscr{O}}

\newcommand{\EExt}{\mathscr{E}xt}
\newcommand{\HHom}{\mathscr{H}om}
\newcommand{\TTor}{\mathscr{T}or}

\newcommand{\sE}{\mathscr{E}}
\newcommand{\EE}{\sE}

\newcommand{\sH}{\mathscr{H}}
\newcommand{\sF}{\mathscr{F}}
\newcommand{\sK}{\mathscr{K}}
\newcommand{\sL}{\mathscr{L}}
\newcommand{\sU}{\mathscr{U}}
\newcommand{\sR}{\mathscr{R}}

\newcommand{\sC}{\mathscr{C}}

\newcommand{\sQ}{\mathscr{Q}}
\newcommand{\cV}{\mathscr{V}}
\newcommand{\sI}{\mathscr{I}}
\newcommand{\cN}{\mathscr{N}}

\newcommand{\sT}{\mathscr{T}}
\newcommand{\spi}{\mathscr{S}}

\DeclareMathOperator{\SL}{{\rm SL}}
\DeclareMathOperator{\id}{{\rm id}}
\DeclareMathOperator{\GL}{{\rm GL}}
\DeclareMathOperator{\Spin}{{\rm Spin}}
\DeclareMathOperator{\Sp}{{\rm Sp}}
\DeclareMathOperator{\rO}{{\rm O}}
\DeclareMathOperator{\s}{{\rm S}}
\DeclareMathOperator{\rk}{rk}

\DeclareMathOperator{\Ext}{Ext}

\DeclareMathOperator{\Hom}{Hom}
\DeclareMathOperator{\im}{Im}

\DeclareMathOperator{\coker}{coker}
\DeclareMathOperator{\HH}{H}
\DeclareMathOperator{\hh}{h}

\DeclareMathOperator{\Mo}{\mathrm{M}}
\DeclareMathOperator{\MI}{\mathrm{MI}}
\DeclareMathOperator{\DD}{\mathscr{D}}
\DeclareMathOperator{\Db}{D^b}
\DeclareMathOperator{\Pic}{Pic}

\DeclareMathOperator{\rP}{P}

\newcommand{\Z}{\mathbb Z}
\newcommand{\C}{\mathbb C}

\newcommand{\p}{\mathbb P}

\newcommand{\pd}{\check{\mathbb P}}
\newcommand{\G}{\mathbb{G}}
\newcommand{\bG}{\mathbf{G}}
\newcommand{\rG}{\mathrm{G}}
\newcommand{\LL}{\mathrm{L}}
\newcommand{\rL}{\mathrm{L}}

\newcommand{\RR}{\mathbold{R}}

\DeclareMathOperator{\ts}{\otimes}
\newcommand{\mono}{\hookrightarrow}

\newcommand{\xr}{\xrightarrow}

\newcommand{\ph}{\mathbf{\Phi}}
\newcommand{\phs}{\mathbf{\Phi^{*}}}
\newcommand{\phx}{\mathbf{\Phi^{!}}}
\newcommand{\eX}{{r_X}} \newcommand{\rX}{q_X}
\newcommand{\tra}{{\mathrm{t}}}
\newcommand{\tth}{{\mathrm{th}}}

\SelectTips{cm}{} \newdir{ >}{{}*!/-5pt/@{>}}
\allowdisplaybreaks[1]
\begin{document}


\title[Even and odd instantons on Fano threefolds]{Even and odd instanton bundles on \\ Fano threefolds of Picard number one}

\author{Daniele Faenzi}
\email{{\tt daniele.faenzi@univ-pau.fr}}
\address{Universit\'e de Pau et des Pays de l'Adour \\
 Avenue de l'Universit\'e - BP 576 - 64012 PAU Cedex - France}
\urladdr{{\url{http://univ-pau.fr/~faenzi/}}}

\thanks{The author was partially supported by ANR-09-JCJC-0097-0
  INTERLOW and ANR GEOLMI}
 \keywords{Instanton bundles, monads, Fano threefolds, semiorthogonal decomposition, nets of quadrics}
 \subjclass[2000]{Primary 14J60, 14F05.}


\sloppy

\begin{abstract}
  We consider an analogue of the notion of instanton bundle on
  the projective 3-space, consisting of a class of rank-$2$ vector
  bundles defined on smooth Fano threefolds $X$ of Picard number one,
  having even or odd determinant according to the parity of $K_X$.

  We first construct a well-behaved
  irreducible component of their moduli spaces.
  Then, when the intermediate Jacobian of $X$ is trivial, we look at
  the associated monads, hyperdeterminants and nets of quadrics.
  We also study one case where the intermediate
  Jacobian of $X$ is non-trivial, namely when $X$ is the intersection of two quadrics in $\p^5$,
  relating instanton bundles on $X$ to vector bundles of higher rank
  on a the curve of genus $2$ naturally associated with $X$.
\end{abstract}

\maketitle

\tableofcontents

\section*{Introduction}

Let $X$ be a smooth complex projective
threefold of Picard number $1$.
Let $H_X$ be the ample generator of the Picard group, and write:
\[
K_X = -i_X H_X,
\]
for some integer $i_X$.
Let us assume that $i_X$ is positive, so that $X$ is a {\it Fano} threefold.
Under these assumptions, each of the cohomology groups $\HH^{p,p}(X)$
is $1$-dimensional, more specifically $\HH^{2,2}(X)$ is generated by
the class $L_X$ 
of a line contained in $X$ (we will thus write the Chern classes of a
coherent sheaf on $X$ as integers).
We have $i_X \in \{1,2,3,4\}$, and $i_X=4$ implies $X=\p^3$, while
$i_X=3$ implies that $X$ is a smooth quadric hypersurface in $\p^4$.
In case $i_X=2$, the variety $X$ is called a {\it Del Pezzo} threefold,
while for $i_X=1$ the variety $X$ is called a
{\it prime} Fano threefold, and in this case its {\it genus $g_X$}
is defined as the genus of a
general double hyperplane section curve.

\medskip

Let now $F$ be a rank-$2$ algebraic vector bundle on $X$, and suppose that
$F$ is stable (with respect to $H_X$). Let us assume that the following conditions hold:
\[
  F \simeq F^* \otimes \omega_X, \qquad \HH^1(X,F)=0.
\]
Set $E$ for the twist $F(t)$ of $F$ having $c_1(E)=0$ or $-1$, and let
$c_2(E)=k$.
Clearly, when $i_X=4$, i.e. when $X=\p^3$, the sheaf $E=F(2)$ is an
instanton bundle on $\p^3$, as first described in
\cite{adhm}. 
By this reason we say that $E$ is a $k$-instanton on $X$.
We say that $E$ is an {\it even} or {\it odd instanton} depending on
the parity of $c_1(E)$, and we 
denote the (coarse) moduli space of these bundles by $\MI_X(k)$.
On $\p^3$, the moduli space of instanton bundles has long been
conjectured to be 
smooth and irreducible, and smoothness has indeed been announced recently
(see \cite{jardim-verbitsky}) together with irreducibility at least
for odd $c_2$ (see \cite{tikhomirov:odd}).
Also, irreducible instanton moduli have been recently proved to be
rational, see \cite{markushevich-tikhomirov:rationality}.

The purpose of this paper is to carry out a preliminary study of the
moduli space of 
$k$-instantons $E$ on Fano threefolds $X$ besides $\p^3$.
Our first result, based on previous work contained in
\cite{brafa1:arxiv}, is to extend the non-emptiness statement
for $\MI_X(k)$ to almost all Fano threefolds of Picard number $1$.
Namely, it holds for any $X$ if $i_X>1$ and, when $i_X=1$, under the
assumption that $-K_X$ is very ample (so $X$ is 
{\it non-hyperelliptic}) and that $X$ contains a line $L$ with normal
bundle $\OO_L \oplus \OO_L(-1)$ (so $X$ is {\it ordinary}).
\begin{THM} \label{esiste}
  The moduli space $\MI_X(k)$ has a generically smooth
  irreducible component whose dimension is the number $\delta$ below:
  \[
  \begin{tabular}[h]{c||c|c|c|c}
    $i_X$ & 4 & 3 & 2 & 1 \\
    \hline
    \hline
    $\delta$ & $8k-3$ & $6k-6$ & $4k-3$ & $2k-g_X-2$
  \end{tabular}
  \]
  and $\MI_X(k)$ is empty when $i_X=2,k=1$ and when $i_X=1, 2k<g_X+2$.
\end{THM}

Next, we focus on the case when the variety $X$ satisfies
$\HH^3(X)=0$, i.e. the intermediate Jacobian of $X$ is trivial.
This holds if the derived category $\Db(X)$ of coherent sheaves on $X$
admits a full
strongly exceptional collection, and a case-by-case analysis shows
that in fact the two conditions are equivalent.
Indeed, there are only $4$ classes of such varieties, one for each index, namely:
\begin{enumerate}[i)]
\item the projective space $\p^3$, for $i_X = 4$;
\item a  quadric hypersurface in $\p^4$, for $i_X = 3$;
\item a linear section $X=\p^6 \cap \bG(2,\C^5)\subset \p^9$, with $H_X^3=5$, for $i_X = 2$;
\item a prime Fano threefold $X \subset \p^{13}$ of genus $12$, in case $i_X = 1$.
\end{enumerate}
In all these case, we have:
\begin{equation}
  \label{Db}
\Db(X) = \big\langle \sE_0,\sE_1,\sE_2, \EE_3 \big\rangle,  
\end{equation}
for some vector bundles $\sE_i$, with $i=0,1,2,3$, defined on $X$.
Let $\rX$ and $\eX$, be the quotient and the remainder of
the division of $i_X$ by $2$.
It turns out that one can choose the $\EE_i$'s so that:
\begin{equation}
  \label{E0}
\EE_0 \simeq \OO_X(-\rX-\eX), \qquad \EE_3^*(-\eX) \simeq \EE_1, \qquad \EE_2^*(-\eX) \simeq \EE_2.  
\end{equation}
Set $U=\Hom_X(\sE_2,\sE_3)$, and note that $U \simeq \Hom_X(\EE_1,\EE_2)$.

Now, given an integer $k$,
we fix vector spaces $I$ and $W$, and an isomorphism $D : W \to W^*$ with
$D^\tra=(-1)^{\eX+1} D$ (an $(\eX +1)$-symmetric duality).
According to the values of $i_X$ and $k$, we need to choose the
dimension of $I$ and $W$ as follows:
\begin{equation}
  \label{valori}
    \begin{tabular}[h]{c|c|c|c}
    $i_X$ & $k$ & $\dim(I)$ & $\dim(W)$ \\
    \hline
    \hline
    $4$ & $k \ge 1$  & $k  $ & $2k+2 $ \\
    $3$ & $k \ge 2$  & $k-1$ & $k    $ \\
    $2$ & $k \ge 2$  & $k  $ & $4k+2 $ \\
    $1$ & $k \ge 8$  & $k-7$ & $3k-20$ 
  \end{tabular}
\end{equation}

The lower bounds for $k$ appear in order to ensure non-emptiness of $\MI_X(k)$.
Let us write $\rG(W,D)$ for the symplectic group $\Sp(W,D)$, or for the
orthogonal group $\rO(W,D)$ depending on
whether $\eX = 0$ or $1$, so that $\eta \in \rG(W,D)$ operates on $W$
and satisfies:
\[
\eta^\tra D \eta = D.
\]
We look at an element $A$ of $I \ts W \ts U$ as a map:
\[
A : W^* \ts \sE_2 \to I  \ts \sE_3,
\]
and, under the dualities \eqref{E0}, we can consider:
\[
D \, A^\tra : I^{*} \ts \sE_1 \to  W^* \ts \sE_2.
\]
We define the subvariety $\sQ_{X,k}$ of $I \ts  W \ts U$ by:
\[
\sQ_{X,k} = \{A \in I \ts  W \ts U \mid A \, D \, A^\tra = 0\},
\]
and its open piece:
\[
\sQ_{X,k}^\circ = \{A \in \sQ_{X,k} \mid \mbox{$A : W^* \ts \sE_2 \to I  \ts \sE_3$ is surjective}\}.
\]
We also define the group:
\[
G_k = \GL(I)  \times \rG(W,D),
\]
acting on $I \ts  W \ts U$ by $(\xi,\eta).A=(\xi^{-1} A \eta^\tra)$.
The group $G_k$ acts on the variety $\sQ_{X,k}$, since,  for all $A \in
\sQ_{X,k}$, we have $\xi^{-1} A \eta^\tra D
\eta A^\tra \xi^{-\tra}=\xi^{-1} A D A^\tra \xi^{-\tra}=0$.
Clearly, $G_k$ acts also on $\sQ_{X,k}^\circ$.
Here is the first main result of this paper.

\begin{THM} \label{generale}
  Let $X$ be a smooth complex Fano threefold of Picard number $1$ and
  $\HH^3(X)=0$.
  Let $I,W,D$ and the $\EE_i$'s be as above.
  Then a $k$-instanton $E$ on $X$ is the cohomology of a monad of the
  form:
  \[
  I^* \ts \EE_1 \xr{D \, A^\tra} W^* \ts \EE_2 \xr{A} I \ts \EE_3,
  \]
  and conversely the cohomology of such a monad is a $k$-instanton.
  The moduli space $\MI_X(k)$ is isomorphic to the geometric quotient:
  \[
  \sQ_{X,k}^\circ / G_k.
  \]
\end{THM}

More specific results in different directions will be given for each
threefold, according to the index, leaving aside the well-known case
of $\p^3$:
\begin{itemize}
\item for $i_X=3$,
i.e. when $X$ is a quadric threefold $Q$, Theorem \ref{quadric} will show that
$\MI_{Q}(k)$ is affine, in analogy with the case of $\p^n$ proved in
\cite{costa-ottaviani:multidimensional}.
We also say some words on jumping lines in this case. Moreover,
we show that 't Hooft bundles provide unobstructed instantons;
\item for $i_X=2$ a description of $\MI_X(k)$ in terms of nets of
  quadrics will be given in Theorem \ref{V5};
\item for $i_X=1$, we provide a parametrization of $\MI_X(k)$ in terms of nets
  of quadrics, describe $\MI_X(k)$ when $k=8$ (i.e. $1$
higher than the minimal value $7$), and we study the variety of
jumping conics of an instanton. We refer to Theorem \ref{X12}.
\end{itemize}

When $\HH^3(X)\ne 0$, there are not enough exceptional objects on $X$
to generate $\Db(X)$.
However in some cases, notably when $X$ is rational, one can rely on
a semiorthogonal decomposition of $\Db(X)$ containing a subcategory
equivalent to $\Db(\Gamma)$, where $\Gamma$ is a curve whose
Jacobian is isomorphic to the intermediate Jacobian of $X$.
We investigated $\MI_X(k)$ in terms of this curve already in
\cite{brafa1:arxiv} (for threefolds of index $1$ and genus $7$) and in 
\cite{brafa3:doi} (for index $1$ and genus $9$).
Here we carry out a similar analysis in the case of Del Pezzo
threefolds of degree $4$, see Theorem \ref{dP4}.

\medskip

Concerning smoothness and irreducibility of $\MI_X(k)$, 
it might be natural to conjecture that these properties hold when $X$ is
general in its moduli space.
Although we do not tackle here this problem, a few comments are in order.
First, in some cases these properties do hold, in particular for low
values of $k$.
For instance, this is the case for $i_X=3$
(i.e. $X$ is a quadric threefold) and $k=2,3$ (see \cite{ottaviani-szurek}),
for most Del Pezzo threefolds (i.e. $i_X=2$) when $k=2$, (see
\cite{markushevich-tikhomirov,markushevich-tikhomirov:double-solid,dani:v5}),
 and for many prime Fano threefolds
(i.e. $i_X=1$) when $\frac{g_X}{2}+1 \le k \le \frac{g_X+1}{2}+2$
(see for instance \cite{iliev-markushevich:degree-quartic, iliev-manivel:genus-8, brafa1:arxiv,
  brafa3:doi, brafa2, iliev-markushevich:sing-theta:asian}).
We contribute to this with Theorem \ref{dP4}, and Theorem \ref{X12}
respectively for Del Pezzo threefolds of degree $4$ and prime Fano
threefolds of genus $12$.

However, it should be
clear that these properties do
not necessarily hold when $X$ is not general in its moduli space.
For instance, for $i_X=1$ and
$g_X=5$ (i.e. $X$ is the intersection of $3$ quadrics in $\p^6$), the
moduli space $\MI_X(4)$ is isomorphic to a double cover of the
discriminant septic, as proved in \cite{brafa2}.
For special $X$, this septic can be singular
and can have many irreducible components.
Examples of threefolds $X$ with $i_X=1$ and
$g_X=7$ such that $\MI_X(6)$ is singular are given in \cite{brafa1:arxiv}.
Still $\MI_X(6)$ is always connected in this case.
Finally, A. Langer outlined 
an argument based on
\cite{langer:quadrics} that suggests that $\MI_X(k)$ cannot be smooth
and irreducible for all $k$ when $X$ is a smooth quadric threefold.

\medskip

Finally, I would like to mention the independent paper of
Alexander Kuznetsov, \cite{kuznetsov:instanton}, appeared shortly after submission of this
manuscript.
Kuznetsov's paper takes up many questions considered here; 
the results are particularly close for Fano threefolds of index $2$
and degree $4$ and $5$.
An interesting analysis of instantons in terms of their jumping lines is
also carried out in {\it loc. cit.}

\medskip

The paper is organized as follows.
The next section contains some basic material, and the definition and
first properties of even and odd instanton bundles on Fano threefolds.
Next, we study $\MI_X(k)$ according to the value
of the index $i_X$.
In Section \ref{i3} we study (odd) instantons on a quadric threefold,
i.e. when $i_X=3$.
In Section \ref{i2} we deal with Del Pezzo threefolds, i.e. with the
case $i_X=2$.
We first give an existence result for instantons on any
such threefold.
Then, we focus on Del Pezzo threefolds of degree $5$, in this case 
we have $\HH^3(X)=0$ and $\MI_X(k)$ is studied in terms of monads.
Next we look at degree $4$, in this case $\HH^3(X)\ne 0$ and
$\MI_X(k)$ is studied in terms of a curve of genus $2$.
In Section \ref{i1} we study $\MI_X(k)$ for $i_X=1$, mainly focusing on
the case prime Fano threefolds of genus $12$
(corresponding to the condition $\HH^3(X)=0$).
In Section \ref{section-quotient} we carry out some basic
considerations to check that the quotient describing $\MI_X(k)$ is geometric.

\section{Basic facts about instantons on Fano threefolds}

\numberwithin{equation}{section}

\subsection{Preliminaries}

We set up here some preliminary notions, mainly concerning Fano
threefolds $X$ of Picard number $1$ with $\HH^3(X)=0$.

\subsubsection{Notations and conventions}

We work over the field $\C$ of complex numbers.
Given a vector space $V$, we denote by $\G(k,V)$ the Grassmann variety
of $k$-dimensional quotients of $V$, and we write $\p(V) = \G(1,V)$.
If $\dim(V)=n$, we write $\bG(k,V)=\G(n-k,V)$.

Let $X$ be a smooth projective variety over $\C$.
We already used the notation $\HH^{k,k}(X)$ for the Hodge cohomology
groups of $X$, and, given a coherent sheaf $F$ on $X$, we let $c_k(F) \in
\HH^{k,k}(X)$ be the $k^\tth$ Chern class of $F$ (we said that $c_k(F)$
 is denoted by an integer as soon as $\HH^{k,k}(X)$ is $1$-dimensional
 and a generator is fixed).
We denote by $\HH^k(X,F)$, or simply by $\HH^k(F)$, the $k^\tth$
cohomology group of $F$ and by  
$\hh^k(X,F)$ its dimension as a $\C$-vector space (the same convention
will be applied to $\Ext$ groups).

\medskip

We refer to \cite{huybrechts-lehn:moduli} for the notion of
stability of a coherent sheaf $F$ on $X$ with respect to a polarization $H$ on
$X$.
Stability, for us, will be in the sense of Maruyama-Gieseker, we
will talk about slope-stability when referring to the terminology of
Mumford-Takemoto.
If $\Pic(X) \simeq \Z$, the choice of $H$ does not matter anymore
and we will avoid mentioning it.
We will denote by $\Mo_{X}(r,c_1,\ldots,c_s)$ the coarse moduli space of
equivalence classes of semistable sheaves on $X$ or rank $r$, with Chern classes
$c_1,\ldots,c_s$, where $s$ is the greatest non-vanishing Chern class
(so $s\leq \dim(X)$).
If $H$ is an ample divisor class on $X$, the degree of $X$
(with respect to $H$) is $H^{\dim(X)}$.
The Hilbert scheme of connected Cohen-Macaulay curves of genus $g$ and degree
$d$ contained in $X$ will be denoted by $\sH^d_g(X)$.
The ideal and normal sheaves of a subscheme $Z \subset X$ will be
denoted by $\sI_{Z,X}$ and $\cN_{Z,X}$. 

We will use the derived category $\Db(X)$ of coherent sheaves on $X$. 
We refer to \cite{huybrechts:fourier-mukai} for basic material on this
category.
We refer to \cite{bondal:helices,gorodentsev:exceptional} for the
notions of simple and exceptional bundle, full (strongly) exceptional collection in $\Db(X)$, 
(Koszul) dual collections, left and right mutations and helices.

\subsubsection{Fano threefolds admitting a full exceptional collection}

We denote by $X$ a smooth Fano threefold of Picard number one, and by
$i_X$ its index, so that $K_X=-i_X H_X$. We have set
$i_X=2\rX+\eX$, where $\rX \ge 0$ and $0 \le \eX \le
1$. 

It is well-known that $X$ admits
a full exceptional collection if and only if $\HH^3(X)=0$
i.e., $X$ has trivial intermediate Jacobian.
Further, the collection in this case is strongly exceptional.
Up to our knowledge, the only available proof of this fact is a
case-by-case analysis. Indeed, recall from the introduction that there
are only $4$ classes of
such varieties, namely:
\begin{enumerate}[i)]
\item the projective space $\p^3$, for $i_X = 4$;
\item a quadric hypersurface in $\p^4$, for $i_X = 3$;
\item a linear section $X=\p^6 \cap \bG(2,\C^5)\subset \p^9$, with $H_X^3=5$, for $i_X = 2$;
\item a prime Fano threefold $X \subset \p^{13}$ of genus $12$, in case $i_X = 1$.
\end{enumerate}

So, for $i_X=4$, this is Beilinson's theorem
\cite{beilinson:derived-and-linear}.
For $i_X=3$, the result is due to Kapranov
\cite{kapranov:derived-homogeneous}.
For $i_X=2$ we refer to \cite{orlov:V5}, see also \cite{dani:v5}.
Finally in case 
$i_X=1$ we refer to \cite{kuznetsov:V22, kuznetsov:V22-preprint, faenzi:v22}.
For each of these cases, 
there is a full strongly exceptional collection $\Db(X)= \big\langle
\sE_0,\sE_1,\sE_2, \sE_3 \big\rangle$, where the $\sE_i$'s are vector
bundles satisfying \eqref{E0}.
We let 
$\big\langle \sF_0,\sF_1,\sF_2, \sF_3 \big\rangle$ be the dual collection.
Then, given a coherent sheaf $\sF$ on $X$, we will write the 
cohomology table of $\sF$
as the following display:
\begin{equation}
  \label{table}
\begin{array}[h]{c|c|c }
\hh^3(\sF\ts \sF_0) & \cdots & \hh^3(\sF \ts \sF_3) \\
\hline
\vdots && \vdots \\
\hline  
\hh^0(\sF\ts \sF_0) & \cdots & \hh^0(\sF \ts \sF_3)  \\
\hline
\hline
\sE_0 & \cdots & \sE_3  \\
\end{array}
\end{equation}
The line containing the $\EE_i$'s here is displayed to facilitate
writing the Beilinson spectral sequence that converges to $\sF$.
This amounts to saying that $\sF$ is the cohomology of a complex
$\sC_\sF^\cdot$ such that: 
\begin{equation}
  \label{complex}
  \sC_\sF^j = \oplus_i
     \HH^{i}(X,\sF \ts \sF_{j-i+3})\ts \sE_{j-i+3}.   
\end{equation}
where the index $i$ in the sum runs between $\max(0,j)$ and $\min(3,j+3)$.

\subsection{The basic orthogonality relation}

Let $F$ be a rank $2$ algebraic vector bundle on $X$, and suppose that
$F$ is stable. Let us assume that the following conditions hold:
\begin{equation}
  \label{instanton-condition}
  F \simeq F^* \otimes \omega_X, \qquad \HH^1(X,F)=0.
\end{equation}
The normalization $E$ of $F$ is defined as the twist $E=F(\rX)$, so:
\[c_1(E) \in \{-1,0\}.\]

\begin{dfn} \label{instanton}
We say that $E$ is an {\it instanton bundle} on $X$ if $E$ is the
normalization of a stable rank-$2$ bundle $F$ on $X$ satisfying
\eqref{instanton-condition}.  
We say that $E$ is a {\it $k$-instanton} if moreover $c_2(E)=k$
(recall that this means $c_2(E)=k L_X$),
and that $E$ is {\it even} or {\it odd} according to whether $c_1(E)$ is even or odd.
We denote by $\MI_{X}(k)$ the subscheme of $\Mo_{X}(2,c_1(E),k)$
consisting of (even or odd) instanton bundles.
\end{dfn}

The following elementary lemma summarizes the cohomology vanishing
satisfied by an instanton bundle.

\begin{lem} \label{nulli}
  If $E$ is an (even or odd) instanton on $X$, then we have the vanishing:
  \begin{equation}
    \label{k}
  \HH^i(X,E(-\rX))=0, \qquad \mbox{for all  $i$},
  \end{equation}
  and also $\HH^1(X,E(-\rX-t))=0$, and $\HH^2(X,E(-\rX+t))=0$,
  for all $t \ge 0$.
  Further, we have the relations in $\Db(X)$:
  \begin{equation}
    \label{orto}
    E \in \langle \OO_X(\rX) \rangle^\bot, \qquad   E \in
    {}^\bot\langle \OO_X(-\rX-\eX) \rangle = \big\langle \sE_1,\sE_2, \EE_3 \big\rangle.
  \end{equation}
\end{lem}

\begin{proof}
  We have defined $E$ as the normalization $E = F(\rX)$, so
  $\HH^1(X,E(-\rX))=0$ by definition. Clearly, we have
  $\HH^0(X,E(-\rX))=0$ by stability of $E$.
  Using Serre duality, the relation $i_X=2\rX+\eX$, and
  the fact that $E^*$ is isomorphic to $E(\eX)$, we get:
  \[
  \HH^i(X,E(-\rX)) \simeq \HH^{3-i}(X,E(-\rX))^*,
  \]
  so \eqref{k} holds for $i=3,2$ as well. This proves
  $\Ext^i_X(\OO_X(\rX),E)$ for all $i$, i.e. the first part of
  \eqref{orto}. For the second one, one just needs to observe that 
  $\Ext^i_X(\OO_X(\rX),E)$ is dual to
  $\Ext^{3-i}_X(E,\OO_X(-\rX-\eX))$.
  The equality ${}^\bot\langle \OO_X(-\rX-\eX) \rangle =
  \big\langle \sE_1,\sE_2, \EE_3 \big\rangle$ is well-known, and has been mentioned
  already in the introduction, see \eqref{E0}.
  However, for the reader's convenience, we will review it in the sequel in a case-by-case fashion.

  Concerning the vanishing $\HH^1(X,E(-\rX-t))=0$, and
  $\HH^2(X,E(-\rX+t))=0$, first note that the second is given by
  the first by the duality argument we already used, and remark that
  they both hold for $t=0$ in view of \eqref{k}.
  To prove the first vanishing for positive $t$, we consider the restriction $E|_{S}$ of $E$ to a general
  hyperplane section $S$ of $X$ and we recall that $E|_{S}$ is slope-semistable
  by Maruyama's theorem, \cite{maruyama:boundedness-small}. This implies $\HH^0(S,E|_{S}(-\rX-t))=0$ which
  gives $\HH^1(X,E(-\rX-t-1)) \subset \HH^1(X,E(-\rX-t))$.
  So we deduce the desired vanishing for any $t$  from the case
  $t=0$, that we have already proved.
\end{proof}

We will use this lemma to show that $E$ is the cohomology of a {\it monad},
i.e. a complex:
\begin{equation}
  \label{monad-general}
  I^* \ts \EE_1 \to W^* \ts \EE_2 \to I \ts \EE_3,  
\end{equation}
where $I$ and $W$ are vector spaces,
the first map is injective and the last map is surjective.

\subsection{Lower bound on the second Chern class of instantons}

The following lemma is almost entirely
well-known, and proves the emptiness statement in Theorem \ref{esiste}.

\begin{lem} \label{empty}
  The moduli space $\Mo_X(2,-\eX,k)$ is:
  \begin{enumerate}[1)]
  \item empty if $i_X=1$, $2k<g_X+3$, and reduced to a single sheaf
    lying in $\MI_X(k)$ if $i_X=1$, $2k=g_X+3$;
  \item reduced to the class of $\OO_X^2$ if $i_X=2$, $k=0$, and
    empty if $i_X=2$, $k=1$.
  \item empty if $i_X=3$, $k=0$, and reduced to a single sheaf
    $\spi$ if $i_X=3$, $k=1$.
  \end{enumerate}
\end{lem}

\begin{proof}
  The emptiness result for $i_X=1$ appears in \cite{brafa1:arxiv}.
  Uniqueness for $i_X=1$ is essentially due to Mukai, however a
  proof can be found in \cite{brafa2}.

  We sketch the argument for the case $i_X=2$.
  Let $S$ be a general hyperplane section of $X$, and recall that $S$
  is a smooth Del Pezzo surface, polarized by $H_S=-K_S$.
  Let $E$ be any slope-semistable sheaf on $X$ with $c_1(E)=0$ and
  $c_2(E)=k$. Set $E_S=E|_{S}$, and recall
  that $E_S$ is slope-semistable (with respect to $H_S=H_X \cap S$)
  by Maruyamas's theorem.
  Note that this easily implies 
  $\HH^2(S,E_S)=0$ by Serre duality, indeed 
  $\HH^2(S,E_S) \simeq \Hom_S(E_S,\OO_S(-1))^*$, and the image of a
  non-zero morphism $E_S \to \OO_S(-1)$ would destabilize $E_S$.

  Now, for $k\le 1$, Riemann-Roch-Hirzebruch gives $\chi(E_S)>0$, so there
  is a non-zero global section of $E_S$, hence $E_S$ cannot be slope-stable.
  Then the Harder-Narasimhan (slope) filtration of $E_S$ reads:
  \[
  0 \to E_1 \to E_S \to E_2 \to 0,
  \]
  where each of the $E_i$'s is a torsion-free sheaf of rank $1$ on
  $S$ with $c_1(E_i)=0$.
  So each $E_i$ is the ideal sheaf of $c_2(E_i)$
  points of $S$.
  Since $k \le 1$, we have $c_2(E_i) \in \{0,1\}$ for $i=1,2$.
  Then, either $E_S$ is $\OO_S^2$ (if $k=0$)
  or one of the $E_i$'s is $\OO_S$ and the other is $\sI_x$,
  where $x$ is a point of $S$.
  When $k=0$, it is easy to deduce $E \simeq \OO_X^2$.
  When $k=1$, if $E_1 \simeq \sI_x$, then $E_S \simeq \OO_S \oplus \sI_x$ because
  $\HH^1(S,\sI_x)=0$.
  Then, no case allows $E_S$ to be semistable (in the sense
  of Maruyama-Gieseker), so $E$ cannot be semistable either.

  The uniqueness result for $i_X=3$, i.e. for the quadric threefolds,
  is very well known, the unique instanton in this case being the
  spinor bundle, see e.g. \cite{ottaviani-szurek}.
  Emptiness for $i_X=3$, $k=0$ is clear by Bogomolov's inequality.
\end{proof}

We note that the above proof show that, when $i_X=2$, an indecomposable
slope-semistable rank-$2$ sheaf $E$ on $X$ with $c_1(E)=0$ and
$c_2(E)=1$ is a vector bundle fitting into:
  \[0 \to \OO_X \to E \to \sI_L \to 0,\]
where $L$ is a line contained in $X$.

\medskip

We will need to use the well-known Hartshorne-Serre correspondence.
Adapting \cite[Theorem 4.1]{hartshorne:stable-reflexive} to our setup, 
for fixed $c_1,c_2$, we have a one-to-one correspondence between the
sets of:
\begin{enumerate}[i)]
\item equivalence classes of pairs $(F,s)$, where $F$ is a rank
  $2$ vector bundle on $X$ with $c_i(F)=c_i$ and $s$ is a global
  section of $F$, up to multiplication by a non-zero scalar, whose
  zero locus has codimension $2$,  
\item locally complete intersection 
curves $Z\subset X$ of degree $c_2$, with $\omega_Z\simeq\OO_Z(c_1-i_X)$.
\end{enumerate}
If the pair $(F,s)$ corresponds to $Z$ in the above bijection, we have:
\[
0 \to \OO_X \xr{s} F \to \sI_Z(c_1) \to 0.
\]
For more details on reflexive sheaves we refer again to \cite{hartshorne:stable-reflexive}.

\section{Odd instantons on a smooth three-dimensional quadric}
\label{i3}

Throughout this section, the ambient threefold $X$ will be a smooth 
$3$-dimensional quadric hypersurface in $\p^4$, and we will denote it
by $Q$.
An (odd) $k$-instanton $E$ on $Q$ is a rank-$2$ stable bundle $E$ with
$c_1(E)=-1$, $c_2(E)=k$ and $\HH^1(Q,E(-1))=0$.
The Hilbert polynomial of $E$ is:
\[
\chi(E(t))=\frac{t+1}3(2t^2+4t-3k+3).
\]

The manifold $Q$ is homogeneous under $\Spin(5)$.
Let $\spi$ be the spinor bundle on $Q$.
We have that $\spi$ is the unique
indecomposable bundle of rank $2$ on $Q$ with $c_1(\spi)=-1$ and
$c_2(\spi)=1$.
Note that $\spi^*(-1) \simeq \spi$.
We set $U=\HH^0(Q,\spi(1)) \simeq \Hom_Q(\spi,\OO_Q)$.
Up to isomorphism, $U$ is the unique indecomposable
$4$-dimensional $\Spin(5)$-module.

Fix an integer $k\geq 2$, a vector space $I$ of dimension $k-1$, a
vector space $W$ of dimension $k$ and a symmetric isomorphism $D : W
\to W^*$.
Recall from the introduction that, in order to study $\MI_Q(k)$, we
are lead to consider the vector space $I \ts W \ts U$, and to view and 
an element $A$ of $I \ts W \ts U$ as a map:
\[
A : W^* \ts \spi \to I  \ts \OO_{Q}.
\]
Transposing $A$ and composing with $D$ we get:
\[
D \, A^\tra : I^{*} \ts \OO_{Q}(-1) \to  W^* \ts \spi.
\]
We write $\sQ_k=\sQ_{Q,k}$ so that $\sQ_{k}=\{A \in I \ts  W \ts U \, \vert \, A \,
D \, A^\tra = 0\}$.
Recall that we set $\rG(W)$ for the orthogonal group $\rO(W,D)$,
and that $G_k=\GL(I) \times \rG(W)$ acts on $\sQ_{k}$ by
$(\xi,\eta).A=\xi^{-1}A \eta^\tra$.

The goal of this section is to prove the following result, that
shows Theorem \ref{esiste} and Theorem \ref{generale} for the case $i_X=3$.
The fact that our moduli space is affine closely resembles the main result
of \cite{costa-ottaviani:multidimensional}.

\begin{THM} \label{quadric}
  The moduli space $\MI_Q(k)$ is the geometric quotient:
  \[
  \sQ_k^\circ/G_k \simeq \{A \in \sQ_k \mid  \DD(A) \neq 0 \}
  /
  G_k
  \]
  where $\DD$ is a
  $G_k\times \Spin(5)$-invariant
  form on $I \ts W \ts U$.
  In particular, $\MI_{Q}(k)$ is affine.
  Moreover, there is a $(6k-6)$-dimensional component of $\MI_Q(k)$,
  smooth along the subvariety consisting of 't Hooft bundles.
\end{THM}
Here, we call {\it 't Hooft (odd) bundle} a rank-$2$ bundle $E$ with $c_1(E)=-1$, such
that $E(1)$ has a global
section that vanishes along the disjoint union of $k$ lines in $Q$.
The proof of the above theorem occupies the rest of this section.

\subsection{Geometry of a smooth quadric threefold}

Let us review some basic facts on the geometry of a smooth
three-dimensional quadric in $\p^4$.
All the material contained in this part is well-known, we refer to
\cite{ein-sols} and to \cite{ottaviani-szurek} for more details.

\subsubsection{The spinor bundle on a quadric threefold}

The manifold $Q \subset \p^{4}=\p(V)$, is homogeneous under the
action of $\Spin(5)$.
The Dynkin diagram of $\Spin(5)$ is of type
$B_{2}$, and the corresponding Lie algebra has two positive roots,
$\alpha_1$ and $\alpha_2$, where $\alpha_1$ is the longest root.
The quadric $Q$ is identified with $\Spin(5)/\rP(\alpha_1)$, and we
have $\p^3=\p(U) \simeq \Spin(5)/\rP(\alpha_2)$.
The group $\Spin(5)$ acts over $V$ by the standard
representation, and over $U$ by the (irreducible) spin representation.
Both $V$ and $U$ are equipped with an invariant duality, so we will silently identify $U$ with
$U^*$ and $V$ with $V^*$.

We denote by $\spi$ the $\Spin(5)$-homogeneous vector bundle given by
the semisimple part of $\rP(\alpha_2)$, normalized to $c_1(\spi)=-1$.
The bundle $\spi$ is called the spinor bundle on
$Q$, it is just the restriction of the universal 
sub-bundle on
the $4$-dimensional quadric, isomorphic to $\G(2,U)$.
Note that the dual of the universal quotient bundle on $\G(2,U)$ also restricts to $\spi$.
We refer to \cite{ottaviani:spinor} for more details on spinor bundles.
We have:
\[
 c_{1}(\spi) = -1, \qquad c_{2}(\spi) = 1, \qquad \HH^i(Q,\spi(t))=0,
\]
for all $t$ if $i=1,2$ and for $i=0$, $t\le 0$.
Further, we have natural isomorphisms:
\[
  \HH^{0}(Q,\spi^*) \simeq U, \qquad \spi^* \simeq \spi(1).
\]
Moreover, we have the universal exact sequence:
\begin{equation} \label{universal-spin}
  0 \to \spi \to U \ts \OO_{Q} \to \spi(1) \to 0.
\end{equation}
We can view $\spi$ also as the bundle associated with a line $L
\subset Q$ via the Hartshorne-Serre correspondence, i.e. we have:
\[
0 \to \OO_Q(-1) \to \spi \to \sI_L \to 0.
\]
The lines of $\p(V)$ contained in $Q$ are parametrized by $\p(U)$.
Writing the above sequence in families with respect to global sections of $\spi$, we
get the description of the universal line $\mathbb{L}$:
\begin{equation}
  \label{uniline}
0 \to \OO_{\p(U)}(-2) \boxtimes \OO_Q(-1) \to \OO_{\p(U)}(-1)
\boxtimes \spi \to \OO_{\p(U)\times Q} \to \OO_{\mathbb{L}} \to 0.
\end{equation}

It is well-known that $\spi$ is stable, and that it is 
characterized by indecomposability and either by Chern classes, or by the
above cohomology vanishing.
For a characterization of $\spi$ in terms of unstable hyperplanes, see
\cite{coanda-faenzi:doi}.

\subsubsection{Derived category of a smooth quadric threefold}

Let us briefly review the structure of the derived category of a
quadric threefold, according to \cite{kapranov:derived-homogeneous}.
The semiorthogonal decomposition \eqref{Db} and the bundles $\sE_i$'s and $\sF_j$'s
satisfying \eqref{E0} reads:
\begin{align}
 \label{col-1} & 
 \Db(Q) = \big\langle \OO_{Q}(-2),  \OO_{Q}(-1), \spi, \OO_{Q} \big\rangle.
\intertext{We have the dual decomposition:}
 \label{col-2} &
 \Db(Q) = \big\langle \OO_{Q}(-1), \sT_{\p^4}(-2)|_{Q}, \spi, \OO_{Q} \big\rangle.
\end{align}

\subsection{The monad of and odd instanton on a quadric threefold}

Let $E$ be an odd instanton on $Q$. Then, according to 
Definition \ref{instanton}, $E$ is a stable bundle of rank $2$ with:
\begin{equation}
  \label{condition-Q}
  c_{1}(E)=-1, \qquad \HH^{1}(Q,E(-1))=0.  
\end{equation}

Let us note incidentally that instanton bundles on quadric
hypersurfaces have been
defined differently by Laura Costa and Rosa Maria Mir\'o-Roig in
\cite{costa-miro:monads-instanton}. 
Our definition
differs from theirs in that they consider bundles with even $c_1$.
Our analysis starts with the next lemma.

\begin{lem} \label{vanishing-Q}
  Let $E$ be a sheaf in $\MI_{Q}(k)$. Then $E$ has the following
  cohomology table with respect to \eqref{col-2}:
  \[
  \begin{array}[h]{c|c|c|c }
    0 & 0 &  0  & 0 \\
    \hline            
    0 & 0 &  0  & 0 \\
    \hline                        
    0 & k-1 &  k  & k-1 \\
    \hline                        
    0 & 0 &  0  & 0 \\
    \hline
    \hline
    \OO_Q(-2) & \OO_Q(-1) & \spi  & \OO_{Q}
  \end{array}
  \]
\end{lem}

\begin{proof}
  For $X=Q$, we have $\rX=1$ and $\eX=1$ so $F=E(1)$.
  The rightmost and leftmost columns are given by stability of $E$ 
  together with Lemma \ref{nulli}
  and by Hirzebruch-Riemann-Roch.
  For the second column from the right,  write the twisted Euler sequence:
  \[
  0 \to \OO_Q(-2) \to \OO_Q(-1)^5 \to \sT_{\p^4}(-2)|_{Q} \to 0,
  \]
  and tensor it by $E$.
  Using Lemma \ref{nulli}, Serre duality,  and $E^*(-1)\simeq E$ we
  get, for all $i$:
  \[
  \HH^{i-1}(Q,E \ts \sT_{\p^4}(-2)|_{Q}) \simeq  \HH^{i}(Q,E(-2)) \simeq \HH^{3-i}(Q,E)^*.
  \]
  These groups vanish for $i=0,2,3$ as we know from the rightmost
  column, and we are left with $\hh^1(Q,E \ts \sT_{\p^4}(-2)|_{Q})=k-1$.

  It remains to look at the third column. We would like to prove:
  \[
     \HH^{i}(Q,E\ts \spi) = 0, \qquad \mbox{for $i \neq 1$},
  \]
  and this will finish the proof, for the dimension of $\HH^{2}(Q,E\ts
  \spi)$ will then be computed again by Hirzebruch-Riemann-Roch.
  The case $i=0$ is clear by stability, and using Serre duality we see
  that the case $i=3$ follows from stability too.
  To show the case $i=2$, we tensor by $E(-1)$ the exact sequence
  \eqref{universal-spin}, and we use $\spi^*(-1) \simeq \spi$ and $\HH^i(Q,E(-1))=0$ for all $i$.
  We get
  $\HH^{2}(Q,\spi \ts E) \simeq \HH^{3}(Q,E\ts \spi(-1))$ and this
  group vanishes by Serre duality and by stability of $\spi$ and $E$.
\end{proof}

\begin{lem} \label{quadric-basic}
  Let $k \ge 2$ and $E$ be a sheaf in $\MI_{Q}(k)$.
  Then $E$ is the cohomology of the following monad:
  \begin{equation}
    \label{eq:monad-quadric}
    I^{*} \ts \OO_{Q}(-1) \xr{D \, A^\tra} W^* \ts \spi \xr{A} I \ts \OO_{Q},
  \end{equation}
  where $I \simeq \C^{k-1}$, $W \simeq \C^{k}$, $D : W \to W^*$ is
  a symmetric duality, and $A$ is surjective.
  Conversely, the cohomology of a monad of this form sits in $\MI_{Q}(k)$.
\end{lem}

\begin{proof}
Let $k \ge 2$ and let $E$ be an odd $k$-instanton on $Q$.
The decomposition of $\Db(Q)$ allows us to write $E$ as cohomology of 
the complex $\sC_E^\cdot$ whose terms are given by \eqref{complex},
once we compute the cohomology table of
$E$ with respect to the collection \eqref{col-1}.
By the previous lemma, this gives that 
$E$ is the cohomology of a monad of the form:
\[
\HH^1(Q,E \ts \sT_{\p^4}(-2)|_{Q}) \ts \OO_Q(-1) \to \HH^1(Q,E \ts \spi) \ts \spi \to
\HH^1(Q,E) \ts \OO_Q.
\]
We set:
\[
I = \HH^1(Q,E), \qquad W = \HH^2(Q,E \ts \spi)^*.
\]
We have computed in the previous lemma $\dim(I) = k-1$, $\dim(W) = k$,
and the proof of the previous lemma gives $\HH^1(Q,E \ts \sT_{\p^4}(-2)|_{Q}) \simeq I^*$.
We can thus rewrite the above resolution as:
\begin{equation}
  \label{eq:first-resolution}
  0 \to I^* \ts \OO_{Q}(-1) \xr{A'} W^* \ts \spi \xr{A} I \ts \OO_{Q} \to 0.
\end{equation}

Now, since $E$ is locally free we have a natural skew-symmetric duality
$\kappa : E \to E^{*}(-1)$. One can easily prove that $\kappa$ must
lift to a skew-symmetric isomorphism between the 
resolution \eqref{eq:first-resolution} and:
\[
  0 \to I^* \ts \OO_{Q}(-1) \xr{A^\tra} W \ts \spi^*(-1) \xr{(A')^\tra} I \ts \OO_{Q} \to 0.
\]
Let $\tilde{\kappa}$ be the isomorphism from the middle term of the
resolution above to the middle term of our original resolution.
Then $\tilde{\kappa}$ is skew-symmetric and lies in:
\[
\HH^{0}(Q,\bigwedge^{2} (W^* \ts \spi^{*}) \ts \OO_Q(-1)) \simeq
\left\{
  \begin{array}[h]{c}
    \HH^{0}(Q,\s^{2}W^* \ts \OO_{Q}) \\
    \oplus \\
    \HH^{0}(Q,\wedge^{2}W^* \ts \s^{2} \spi^{*}(-1)).
  \end{array}
\right.
\]

The second summand in the above decomposition is zero.
On the other hand, we may identify
the first one with $\s^{2}W^*$.
The element $D$ of $\s^{2}W^*$ induced by $\tilde{\kappa}$ thus gives a symmetric
isomorphism $W \to W^*$, and we have $A'=D\, A^\tra$.

For the converse implication, let $E$ be the cohomology sheaf of such
a monad.
A straightforward computation shows that $E$ has the desired Chern
classes.
Further, it is easy to see that $E$ satisfies $\HH^0(Q,E)=0$ for
$\HH^1(Q,\OO_Q(-1))=0$ and $\HH^0(Q,\spi)=0$, so $E$ is stable.
Finally, twisting the monad by
$\OO_Q(-1)$ we immediately see that $E$ satisfies  the cohomology vanishing required to lie in $\MI_Q(k)$.
\end{proof}

Note that, for $k=1$, the above lemma says that $\MI_Q(1)$ consists only of the sheaf $\spi$.
Here we describe the moduli space $\MI_Q(k)$ by rephrasing the non-degeneracy
condition of the map $A$ in terms of an invariant resembling the hyperdeterminant.

\begin{proof}[Proof of Theorem \ref{quadric}, first part]
  We prove here all statements of Theorem \ref{quadric} except the
  part related to 't Hooft bundles.

  Let us first define the form $\DD$.
  Recall that there is a natural identification of $U$ and $U^*$.
  Consider
  an element $A$ of $I \ts W \ts U $ as a linear map:
  \[
  \phi_A : W^* \ts U \to I.
  \]
  Taking the symmetric square of $\phi_A$ we obtain a linear map:
  \begin{equation} \label{sym2A}
  \s^2(\phi_A) : \wedge^{2} W^* \ts \wedge^{2} U \to \s^{2} I,
  \end{equation}
  and recall that $\wedge^{2} U$ contains a unique
  $\Spin(5)$-invariant subspace of rank $1$, generated by a $2$-form
  $\omega$.
  We consider thus $\langle \omega \rangle \subset \wedge^{2} U$,
   and the restriction of
  \eqref{sym2A} to $\wedge^{2} W^* \ts \langle \omega \rangle \to
  \s^{2} I$. This gives a square matrix $M_A$ of order ${{k+1}\choose{2}}$.
  We define:
  \[
  \DD(A) = \det(M_A).
  \]
  Clearly, $\DD$ is $ G_k \times \Spin(5)$-invariant by definition.
  
  \medskip
  In view of the previous lemma, the next thing to do is to show that
  $\sQ_k^\circ=\{A \in \sQ_k \mid \DD(A)\ne 0\}$, i.e. that $A \in \sQ_k$
  is surjective as a map $W^* \ts \spi \to I \ts \OO_{Q}$ if and
  only if $A \in \sQ_k$ satisfies $\DD(A)\ne 0$.

  Let us now show that, that the first condition implies the second.
  So, assume that $A \in \sQ_k$ corresponds to a surjective map $W^* \ts \spi \to I
  \ts \OO_{Q}$,  let $E$ be the odd instanton bundle associated
  with $A$, and let us check that $\DD(A)$ is non-zero.
  So let $E$ be defined by the monad \eqref{eq:monad-quadric}, and set
  $K=\ker (A)$.
  It is easy to check the vanishing of $\HH^{j}(Q,K(-1))$
  for all $j$.

  Considering the symmetric square of
  \eqref{eq:monad-quadric} we obtain the exact sequences:
  \begin{align}
    \label{eq:sym2E_1} & 0 \to \s^{2} I ^{*}(-2) \to K \ts I^{*}(-1)
    \to \wedge^{2} K \to \wedge^{2} E \to 0, \\
    \label{eq:sym2E_2} & 0 \to \wedge^{2} K \to 
    \begin{array}{c}
      \wedge^{2} W^*   \ts \s^{2} \spi \\
      \oplus \\
      \s^{2} W^*   \ts \wedge^{2} \spi
    \end{array}
    \to W^*   \ts I \ts \spi \to \s^{2} I \to 0.
  \end{align}
  
  Since all the cohomology groups of the first, second, and fourth term in
  \eqref{eq:sym2E_1} vanish, so do those of $\wedge^{2} K$. Plugging
  into \eqref{eq:sym2E_2}, since $\HH^{j}(Q,\spi)=0$ for all $j$, and
  since $\wedge^{2} \spi \simeq \OO_{Q}(-1)$ we obtain an isomorphism:
  \begin{equation} \label{another-iso}
  \wedge^{2} W^* \ts \HH^{1}(Q,\s^{2} \spi) \simeq \s^{2} I.
  \end{equation}

  Note that the $1$-dimensional vector space $\HH^{1}(Q,\s^{2} \spi)$
  is isomorphic to $\Ext^{1}_{Q}(\spi^{*},\spi)$, and this extension
  corresponds to the universal exact sequence \eqref{universal-spin}.
  Recall that
  \[
  \wedge^{2} U \simeq V \oplus \langle \omega \rangle,
  \]
  as $\Spin(5)$-modules, and $\HH^{0}(Q,\OO_{Q}(1)) \simeq V$.
  Then, taking cohomology in the symmetric square of
  \eqref{universal-spin} we see that $\HH^{1}(Q,\s^{2} \spi)$ equals
  the kernel of $\wedge^{2} U \to V$.
  Therefore, taking cohomology in \eqref{eq:sym2E_2} and restricting
  to $\wedge^{2} W^* \ts \HH^{1}(Q,\s^{2} \spi)$ (i.e. considering the
  isomorphism of \eqref{another-iso}) amounts to looking at the restriction to
  $\wedge^{2} W^* \ts \langle \omega \rangle$ of $\s^{2} (\phi_A)$ (i.e. the map whose determinant is $\DD(A)$).
  So $\DD(A) \ne 0$.

  \medskip

  Conversely, let us now prove that, if $A$ lies in $\sQ_k$ and $A$ is not surjective,
  then $\DD(A)=0$. 
  This time we prefer to work dually and consider 
  the map $A^\tra$ and its lift $\phi_A^\tra :I^{*} \to W \ts U$.
  A point $[v] \in Q$ is represented by a vector $v
  \in V$ and we have a projection $\pi_{v}:U
  \to \spi_{v}^*$. The assumption that $A$ fails to be surjective
  (i.e., $A^\tra$ is not injective) means that we can find a vector
  $y \in I^{*}$ such that $A_{v}^\tra(y)=0$ at some point $[v] \in
  Q$. Equivalently, $\id_{W} \ts \pi_{v} \, \phi_A^\tra (y)=0$.

  We consider the map $g$ defined as the composition:
  \[
  \s^{2} I^{*}
  \xr{\s^{2} ({\phi_A^\tra})}
  \wedge^{2} W \ts U \ts U
  \xr{\id_{\wedge^{2} W \ts U}  \ts \pi_{v}}
  \wedge^{2} W \ts U \ts \spi^{*}_{v}.
  \]
  Under our degeneracy assumption, we have $g(y^{2})=0$.
  We would like to check that
  $\s^{2}({\phi_A^\tra})(y^{2})$ goes to zero under the projection $\Pi$ of
  $\wedge^2 W \ts U \ts U$ onto $\wedge^{2} W \ts \langle \omega
  \rangle$ (recall that we have an invariant splitting $\wedge^2 U
  \simeq V \oplus \langle \omega \rangle$), for in this case we
  clearly have $\DD(A)=0$.
  Note that we have the factorization:
  \[
  g = \id_{\wedge^{2} W} \ts j_{v} \, \Pi \, \s^{2} ({\phi_A^\tra}),
  \]
  where $j_{v}: \langle \omega \rangle \to U \ts
  \spi^{*}_{v}$ is defined by the symmetric square of
  \eqref{universal-spin}.
  
  Hence we are done if we prove that $j_{v}$ is injective. But the map
  $j_{v}$ is the evaluation at $[v] \in Q$ of the natural map $\OO_{Q}
  \to \HH^{0}(Q,\spi^{*})^* \ts \spi^{*}$, which is a fibrewise
  injective map of vector bundles.

  We have thus proved that $\sQ_k^\circ = \{ A \in \sQ_k \mid \DD(A)
  \ne 0\}$.
  One then shows that the map $\{A \in \sQ_k \mid
  \DD(A)\ne 0\} \to \MI_Q(k)$ that associates with $A$ the cohomology
  of the monad given by $A$ is an affine geometric quotient, in
  the same way as in \cite[Section 2]{costa-ottaviani:multidimensional}.
\end{proof}

This proves the first part of Theorem \ref{quadric}.
In the next section we look at 't Hooft bundles on $Q$,
and give the proof of the second part of Theorem \ref{quadric}.

\subsection{Unobstructedness of odd 't Hooft bundles}

By analogy with the case of $\p^3$, we speak of (odd) 't Hooft bundles
for (odd) instantons associated with a configuration of skew lines
contained in $Q$. 

\begin{dfn}
Let $\rL = L_1 \cup \cdots \cup L_k$ be the union of $k \ge 2$ disjoint
lines in $Q$.
Then by Hartshorne-Serre's construction we have a
rank-$2$ bundle $E_\LL$ fitting into:
\begin{equation} \label{EL}
  0 \to \OO_Q(-1) \to E_\LL \to \sI_{\LL,Q} \to 0.
\end{equation}
The bundle $E_\LL$ is said to be the 't Hooft (odd) bundle associated with $\LL$.
If $\LL$ is contained in (one ruling of) a smooth hyperplane section
of $Q$, then $E_\LL$ is said to be a special 't Hooft bundle.
\end{dfn}

The following proposition shows that $\MI_Q(k)$ is smooth along the
subvariety of 't Hooft bundles, thus completing the proof of Theorem \ref{quadric}.

\begin{prop} \label{thooft}
  Let $\LL \subset Q$ be a union of $k$ disjoint lines, and let
  $E=E_\LL$ be defined as above. Then $E$ lies in $\MI_Q(k)$, and satisfies
  $\Ext^2_Q(E,E)=0$.
\end{prop}

\begin{proof}
  Set $E=E_\LL$. Let us first note that $c_1(E)=-1$ (so
  $E^*(-1) \simeq E$) and $c_2(E)=k$.
  We observe that there are isomorphisms:
  \begin{equation}
    \label{normal}
  E^* \ts \OO_\LL \simeq \cN_{\LL,Q} \simeq
  \bigoplus_{j=1,\ldots,k}\OO_{L_j} \oplus \OO_{L_j}(1).
  \end{equation}
  
  To prove that $E$ satisfies the cohomology vanishing \eqref{condition-Q}, it suffices to take
  cohomology of \eqref{EL}, twisted by $\OO_Q(-1)$, and to observe
  that $\HH^i(Q,\OO_\LL(-1))=0$ for all $i$. Moreover, stability of
  $E$ is easily checked, since $\HH^0(Q,E_\LL)=0$ follows
  immediately from \eqref{EL}.
  
  Let us now prove $\Ext^2_Q(E,E)=0$.
  Applying the functor $\Hom_Q(E,-)$ to the exact
  sequence \eqref{EL}, we obtain:
  \[
  \HH^2(Q,E) \to  \Ext^2_Q(E,E) \to
  \HH^2(Q,E^* \ts \sI_{\LL,Q}).
  \]
  Since $E$ is an odd instanton bundle, we have seen in Lemma
  \ref{vanishing-Q} that $\HH^2(Q,E)=0$.
  So we only need to check $\HH^2(Q,E^* \ts \sI_{\LL,Q})=0$.
  To do this, we take cohomology of the exact sequence:
  \[
  0 \to E^* \ts \sI_{\LL,Q} \to E^* \to E^* \ts \OO_{\LL} \to 0.
  \]
  Using \eqref{normal} we see that $\HH^1(Q,E^* \ts
  \OO_{\LL})=0$ and Lemma \ref{nulli} gives $\HH^2(Q,E^*)=0$.
  This implies the desired vanishing.
\end{proof}

\subsection{Jumping lines of an odd instanton over a quadric threefold}

Let $E$ be an odd instanton on $Q$, and $L$ a line contained in $Q$.
\begin{dfn}
We say that $L$ is {\it jumping} for $E$ if:
\[
\HH^1(L,E|_L) \ne 0.
\]
Note that \cite[Proposition 5.3]{coanda-faenzi:doi} ensures that not all lines contained
in $Q$ are jumping for $E$.
We say that $E$ has {\it generic splitting} if the set of jumping
lines of $E$ has codimension $2$ in $\p(U)$.
We denote by $\sC_E$ the curve of jumping lines of an instanton $E$
with generic splitting, and by $i$ its embedding in $\p(U)$.
\end{dfn}

\begin{prop}
  Let $E$ be a $k$-instanton on $Q$ with generic splitting.
  Then $\sC=\sC_E$ is a Cohen-Macaulay curve of degree ${k \choose 2}$, equipped with a torsion-free sheaf $\sF$ fitting into:
  \[
  0 \to \OO_{\p(U)}(-k) \to W^* \ts \OO_{\p(U)}(-1) \xr{B_E} I \ts \OO_{\p(U)} \to i_*\sF \to 0.
  \]
  The instanton $E$ can be recovered from $(\sC,\sF)$.
\end{prop}

\begin{proof}
  Denote by $p$ and $q$ the projections from $\mathbb{L}$ to $Q$ and $\p(U)$.
  A line $L$ is jumping for $E$ if and only if the point of $\p(U)$
  corresponding to $L$ lies in the support of $\RR^1q_*p^* (E)$.

  Applying $q_*p^*$ to the terms of the monad of $E$ with the help of
  \eqref{uniline}, we get $q_*p^*(\OO_Q)\simeq \OO_{\p(U)}$, 
  $q_*p^*(\OO_Q(-1)) = 0$, and $q_*p^*(\spi)\simeq
  \OO_{\p(U)}(-1)$, where this last isomorphism follows easily from
  $\HH^i(Q,\spi \ts \spi) =0$ for $i \ne 0$ and $\HH^1(Q,\spi \ts
  \spi) \simeq \C$.
  All higher direct images of the terms of the monad vanish.

  Therefore, we get the desired map $B_E$ as $q_*p^*(A)$, and $\RR^1q_*p^*(E)
  \simeq \coker(B_E)$.
  Since $E$ has generic splitting, the support $\sC$ of $\coker(B_E)$
  is determinantal a curve in $\p(U)$, so that $\coker(B_E)$ is the
  extension by zero to $\p(U)$ of a 
  Cohen-Macaulay sheaf $\sF$ on $\sC$ and $\ker(B_E) \simeq
  \OO_{\p(U)}(-k)$. We get thus the desired exact sequence.

  Finally, given $(\sF,\sC)$, we obtain the displayed exact sequence
  by sheafifying the minimal graded free resolution of the graded
  $\C[x_0,\ldots,x_3]$-module
  associated with $i_*\sF$, so
  that $B_E$ is determined by $\sF$ up to conjugacy.
  But $B_E \in \Hom_{\p(U)}(W^* \ts \OO_{\p(U)}(-1),I \ts \OO_{\p(U)})
  \simeq W \ts I \ts U$ is nothing but $A$, so $E$ is reconstructed
  from $(\sF,\sC)$.
  Since $\sC$ is the degeneracy locus of $B_E$ and has the expected
  codimension $2$ in $\p(U)$, we get that $\sC$ is Cohen-Macaulay of degree ${k \choose 2}$.
\end{proof}

\section{Instanton bundles on Del Pezzo threefolds}
\label{i2}

Here we will look at instanton bundles on Fano threefolds of index
$2$, also called Del Pezzo threefolds, still in the assumption of
Picard number one.
According to Iskovskikh's classification, see
\cite{iskovskih:I, iskovskih:II}, consult also \cite{fano-encyclo},
there are $5$ deformation classes of these threefolds, characterized
by the degree $d_X=H_X^3$, that ranges from $1$ to $5$.

Recall that a $k$-instanton bundle $E$ in this case 
is a stable rank-$2$ bundle on $X$ with:
 \begin{equation} \label{instanton-i2}
 c_{1}(E)=0, \qquad c_{2}(E)=k, \qquad \HH^{1}(X,E(-1))=0.
 \end{equation}
The Hilbert polynomial of $E$ is:
\[
\chi(E(t))=\frac{t+1}3 (d t^2+2d t-3k+6).
\]

\subsection{Construction of instanton bundles on Del Pezzo threefolds}

Some properties of the moduli spaces $\Mo_X(2,0,k)$ and $\MI_X(k)$ have been
already investigated.
In \cite{markushevich-tikhomirov}, it is proved that
$\MI_X(2)$ is \'etale over an open subset of the intermediate Jacobian $J(X)$ of $X$ via
the Abel-Jacobi mapping when $d_X=3$,  i.e. when $X$ is a cubic
threefold, and the degree of this cover is one according to \cite{iliev-markushevich:degree-14}.
In fact $\MI_X(2,0,2)$ is a blowup of $J(X)$
along $\sH^1_0(X)$ by \cite{druel:cubic-3-fold}.
Via the same map, by \cite{markushevich-tikhomirov:double-solid},
$\MI_X(3)$ parametrizes the Theta divisor of $J(X)$ when $d=2$.

The aim of this section, however, is to construct a well-behaved
component of $\MI_X(k)$ for all $X$ and for all $k \ge 2$.
This, together with the construction of Proposition \ref{thooft} (for
$i_X=3$) and 
\cite[Theorem 3.9]{brafa1:arxiv} (for $i_X=2$) will achieve the proof
of Theorem \ref{exists}.

\begin{thm} \label{exists}
  Let $X$ be a Fano threefold of index $2$ and $k\ge2$ an integer.
   Then there exists a $(4k-3)$-dimensional, generically smooth component of $\MI_X(k)$.
\end{thm}

\begin{proof}
  Let $X$ be a smooth Fano threefold of index $2$ and degree $d=d_X$,
  so that
  $H_X$ maps $X$ to $\p^{d+1}$ (with a point of indeterminacy if $d=1$).
  We work by induction on $k \ge 2$, following an idea developed in \cite{brafa1:arxiv}.
  We would like to prove, for all $k \ge 2$, that there exists a
  vector bundle $E$ satisfying \eqref{instanton-i2}, and a line $L\subset X$ with:
  \begin{equation}
    \label{smooth-induction}
    \Ext^2_X(E,E)=0, \qquad E|_{L} \simeq \OO_L^2.    
  \end{equation}

  \begin{step} {\it Base of the induction, $k=2$}.
  This case has been studied in several papers, 
  see for instance
  \cite{arrondo-costa,markushevich-tikhomirov:double-solid,markushevich-tikhomirov,dani:v5}.
  We sketch here a uniform argument.
  Let $S$ be a general hyperplane section of $X$, so that
  $S$ is a Del Pezzo surface of degree $d$, given as blowup of $\p^2$
  in $9-d$ points. Let $h$ be the pullback to $S$ of the class of a line
  in $\p^2$, and denote by $e_i$'s the exceptional divisors, hence
  $(H_X)|_{S}=3h-e_1-\cdots-e_{9-d}$.
  
  Take a general curve $C$ of class $3h-e_1-\cdots-e_{7-d}$, so that
  $C$ is a smooth irreducible curve of genus $1$ and degree $d+2$.
  Note that the normal bundle $\cN_{C,X}$ fits into:
  \[
  0 \to \OO_C(C) \to \cN_{C,X} \to \OO_C(H_C) \to 0,
  \]
  where we set $H_C = (H_X)|_{C}$.
  Clearly, we have $\hh^0(C,\OO_C(H_C))=d+2$ and
  $\HH^1(C,\OO_C(H_C))=0$, since $C$ is an elliptic 
  curve of degree $d+2$.
  Also, working in the Del Pezzo surface $S$, since $C$ is
  a smooth irreducible elliptic curve with
  $C^2=9-7+d=d+2$, we have $\hh^0(C,\OO_C(C))=d+2$ and $\HH^1(C,\OO_C(C))=0$.
  By the above exact sequence we get
  $\hh^0(C,\cN_{C,X})=2d+4$ and $\HH^1(C,\cN_{C,X})=0$.
  So the Hilbert scheme $\sH^{d+2}_1(X)$ is smooth of dimension $2d+4$ at the point
  corresponding to $C$. Let us prove that a general deformation of $C$
  in $\sH^{d+2}_1(X)$ is non-degenerate.

  Assuming the contrary, we consider the incidence of pairs $(D,H)$ in $\sH^{d+2}_1(X) \times
  \pd^{d+1}$ such that $D$ lies in $H$, and we may suppose that the
  projection onto (an open neighborhood of the point given by
  $C$) of $\sH^{d+2}_1(X)$ is dominant.
  Note that the
  projection onto $\pd^{d+1}$ is dominant by construction.
  Observe now that the fibre over $C$
  is a single point corresponding to the hyperplane section $S$, so
  the same happens to the fibre over a general $D$ specializing to $C$.
  On the other hand, the general fibre over a point of $\pd^{d+1}$
  corresponding to a hyperplane $H \subset X$ can be
  identified with $\p(\HH^0(H,\OO_H(C)))=\p^{d+2}$.
  The dimension count leads thus to argue that $\sH^{d+2}_1(X)$ should
  be of dimension $2d+3$, locally around the point corresponding to $C$.
  This is a contradiction, so we have proved that 
  a general deformation $D$ of $C$ lies in no hyperplane.

  Now, the correspondence of Hartshorne-Serre associates with $D$ a rank-$2$ vector
  bundle $E$ fitting into:
  \begin{align}
    \label{ED}
    & 0 \to \OO_X(-1) \to E \to \sI_{D,X}(1) \to 0, \\
    \nonumber & 0 \to \sI_{D,X}(1) \to \OO_X(1) \to \OO_{D}(1) \to 0.
  \end{align}

  Computing Chern classes we get $c_1(E)=0$, $c_2(E)=2$.
  Taking cohomology of the above sequences, and of the same 
  sequences twisted by $\OO_X(-1)$,  
  immediately says that $E$ is stable and satisfies \eqref{instanton-i2}.
  To compute $\Ext^2_X(E,E)$, note that this group is isomorphic to
  $\HH^2(X,E \ts E)$ and tensor the sequences in \eqref{ED} by $E$.
  We have  shown $\HH^j(X,E(-1))=0$ for all
  $j$, and it is easy to show that $\HH^2(X,E(1))=0$.
  Therefore, to prove $\HH^2(X,E \ts E)=0$, it suffices to show 
  $\HH^1(D,E(1)|_{D})=0$. But we have $E(1)|_{D} \simeq \cN_{D,X}$, and
  since $D$ is a general deformation of $C$ and
  $\HH^1(C,\cN_{C,X})=0$, we also have $\HH^1(D,\cN_{D,X})=0$.
  We have thus proved $\Ext^2_X(E,E)=0$.

  It remains to find a line $L \subset X$ with $E|_{L} \simeq \OO_L^2$.
  We choose $L$ in $S$ so that $L \cap C$ is a single point $x$.
  Then we have $\sI_{C,X}(1) \ts \OO_L \simeq \OO_x \oplus \OO_L$.
  Therefore, tensoring the first exact sequence of \eqref{ED} by
  $\OO_L$, we get a surjection of $E|_{L}$ to $\OO_x \oplus \OO_L$.
  Therefore $E|_{L}$ cannot be $\OO_L(t) \oplus \OO_L(-t)$ for any
  $t>0$. Hence $E|_{L}$ is trivial. Note that this is an open
  condition on the variety of lines contained in $X$, so it takes
  place for a general line.
  \end{step}

  Let us now take care of the induction step.
  This is very similar to the argument used in 
  \cite[Theorem 3.9]{brafa1:arxiv}.
  We give a proof only accounting for the differences with respect to
  that paper.

  \begin{step}[\it Induction step: defining a non-reflexive sheaf with
    increasing $c_2$]
    We take a $k$-instanton $E$ satisfying \eqref{smooth-induction}
    for a general line $L\subset X$.
    We want to construct a $(k+1)$-
    instanton $F$ again satisfying \eqref{smooth-induction} for some
    line $L \subset X$, a deformation of a sheaf $G$ with
    $c_1(G)=0$, $c_2(G)=k+1$.
    To do this, we choose a projection $\pi : E \to \OO_L$ and we
    define $G = \ker(\pi)$, so we have:
    \begin{equation}
      \label{G}
      0 \to G \to E \to \OO_L \to 0.
    \end{equation}
    It is easy to prove that $G$ is a stable
    sheaf of rank $2$ with $c_1(G)=0$, $c_2(G)=k+1$ and
    $c_3(G)=0$, and using \eqref{G} we see also $\HH^i(X,G(-1))=0$ for
    all $i$.
    Further, applying $\Hom_X(-,G)$ to \eqref{G}, it is not hard to prove $\Ext^2_X(G,G)=0$.
    Then $\Mo_X(2,0,k+1)$ is smooth at the point corresponding to
    $G$, and by a Hirzebruch-Riemann-Roch computation it has dimension $4(k+1)-3
    = 4k+1$.
    Moreover, tensoring \eqref{G} by $\OO_L$ and using
    $\TTor_1^{\OO_X}(\OO_L,\OO_L) \simeq \cN_{L,X}^* \simeq \OO_L^2$ for a
    general choice of $L$, one sees that $\HH^0(X,G \ts \OO_L(-1))=0$.
\end{step}

  \begin{step}[\it Induction step: deforming to a locally free sheaf]
    We let now $F$ be a general deformation of $G$ in
    $\Mo_X(2,0,k+1)$.
    By semicontinuity of Ext sheaves (see
    \cite{banica-putinar-schumacher}), we get that $F$ will satisfy \eqref{instanton-i2},
    $\Ext^2_X(F,F)=0$, and $\HH^0(X,F \ts \OO_L(-1))=0$.
    So it only remains to prove that $F$ is locally free, for then $F$ is a
    $(k+1)$-instanton satisfying \eqref{smooth-induction}.

    To achieve this, one first notes that, for any line $L'\subset X$,
    in a neighborhood in $\sH^1_0(X)$ of $L$,
    and for $E'$ in a neighborhood in $\Mo_X(2,0,k)$ of $E$, the sheaf $F$ cannot fit into an exact
    sequence of the form:
    \begin{equation}
      \label{E'}
      0 \to F \to E' \xr{\pi'} \OO_{L'} \to 0.      
    \end{equation}
    Indeed, the sheaves fitting into such sequence form a family of
    dimension $4k-3+2+1=4k$, indeed we have $4k-3$ choices for $E'$, a
    $2$-dimensional family of lines contained in $X$ for the choice of
    $L'$ and a $\p^1=\p(\HH^0(L',E'|_{L'}))$ for the choice of $\pi'$.
    But $F$ moves in a smooth (open) part of $\Mo_X(2,0,k+1)$ of
    dimension $4k+1$.

    Then, one proves that if $F$ was not locally free, it would have
    to fit into \eqref{E'}, thus finishing the proof.
    To get this, one takes the double dual:
    \begin{equation}
      \label{doubledual}
      0 \to F \to F^{**} \to T \to 0,
    \end{equation}
    where $T$ is a coherent torsion sheaf having dimension $\le 1$.
    What we have to show is now that $T$ takes the form $\OO_{L'}$, for
    some line ${L'}$ contained in $X$, and that $F^{**}$ is a deformation
    of $E$ in $\Mo_X(2,0,k)$.

    Let us first prove $T\simeq \OO_{L'}$ for some line ${L'}\subset X$.
    To achieve this, it suffices to show the three conditions $\HH^0(X,T(-1))=0$, $c_2(T)=-1$ and $c_3(T)=0$.
    Indeed, $c_2(T)=-1$ means that $T$ is a sheaf of generic rank $1$
    on a curve of degree $1$, which is thus a line ${L'}$ contained in
    $X$. The first condition ensures that $T$ is purely
    $1$-dimensional, i.e., $T$ is torsion-free on ${L'}$ and
    the support of $T$ has no isolated points, so that $T \simeq
    \OO_{L'}(t)$ for some $t$. The condition $c_3(T)=0$ then amount to $t=0$. 
    We already have $c_2(T) \le -1$ since $T$ is supported in
    codimension $2$.

    Now, we first see that $\HH^0(X,F^{**}(-1))=0$.
    Indeed, a non-zero global section of $F^{**}(-1)$ would induce,
    via \eqref{doubledual}, a subsheaf $K$ of $F$ having $c_1(K)=1$,
    which contradicts $F$ being stable.
    Then, from $\HH^1(X,F(-1))=0$ we get
    $\HH^0(T(-1))=0$.

    Next comes the semicontinuity argument.
    Take cohomology of \eqref{doubledual} twisted by $\OO_X(t)$,
    and note that $F^{**}$ is reflexive, so $\HH^1(X,F^{**}(t))=0$
    for $t \ll 0$, by \cite[Remark 2.5.1]{hartshorne:stable-reflexive}
    (this remark is formulated for $\p^3$, but it holds in fact for
    any smooth projective threefold).
    We get $\hh^1(X,T(t))
    \le \hh^2(X,F(t))$ for $t\ll 0$.
    Recall that $c=c_3(F^{**}) \ge 0$ since $F^{**}$ is reflexive,
    and note that $c$ and $c_2(T)$ are invariant under twist by
    $\OO_X(t)$, so $c_3(T(t))=-2t c_2(T) + c$.
    Using Hirzebruch-Riemann-Roch we get:
    \[
    \hh^1(X,T(t)) = - \chi(T(t)) = c_2(T)(t+1) - c/2, \qquad \mbox{for $t\ll 0$}.
    \]
    On the other hand, by semicontinuity we have $\hh^2(X,F(t)) \le
    \hh^2(X,G(t)) = \hh^1(X,\OO_{L}(t)) = -(t+1)$ (by \eqref{G}) for $t \ll 0$.
    Summing up, we have:
    \[
    -c_2(T)(t+1) + c/2 \le -(t+1), \qquad \mbox{for $t\ll 0$}.
    \]
    This clearly implies $c_2(T) \ge -1$, hence $c_2(T)=-1$ so by the
    inequality above (and $c\ge 0$) we deduce $c=0$ hence also
    $c_3(T)=0$. This finishes the proof that $T \simeq \OO_{L'}$ for some
    line ${L'}\subset X$.

    To conclude, we have to show that $F^{**}$ is a deformation of $E$
    in $\Mo_X(2,0,k)$.
    But $c_2(T)=-1$ and  $c_3(T)=0$, so $c_2(F^{**})=k$ and
    $c_3(F^{**})=0$, which implies that $F^{**}$ has the same Chern classes as
    $E$.
    Therefore, $F^{**}$ is clearly a flat deformation of $E$, and as
    such is also a semistable sheaf by \cite{maruyama:openness}, so $F^{**}$ lies in a neighborhood
    of $E$ in $\Mo_X(2,0,k)$.
\end{step}

\end{proof}

\subsection{Instanton bundles on Fano threefolds of degree 5}

Here, we focus on the case when $X$ satisfies $\HH^3(X)=0$, in other
words to the case when the derived category of $X$ is finitely
generated, and study instanton bundles in a monad-theoretic fashion.

In this case, the threefold $X$
is obtained fixing a $5$-dimensional vector space $U$,
and cutting $\bG(2,U^*) \subset \p^9$ 
with a $\p^6 \subset \p^9$, i.e.
\[
    X = \bG(2,U^*) \cap \p^6 \subset \p^{9} = \p(\wedge^2 U),
\]
where the $\p^{6}$ is chosen so that $X$ is smooth of
dimension $3$. This threefold is usually denoted as $V_5$, as its
degree $d_X=H_X^3$ is $5$.
The choice of $\p^6=\p(V)$ in $\p^9=\p(\wedge^2 U)$ corresponds to the
choice of a $3$-dimensional subspace $B$ of $\wedge^2 U$, so that
$V=\wedge^2 U/B=\HH^0(X,\OO_X(1))$.
In other words, a $2$-dimensional vector subspace $\Lambda$
of $U^*$ is an element of $X$ if the composition
$\sigma_\Lambda : \Lambda \ts \Lambda \to U^* \ts U^* \to \wedge^2 U^*
\to B^*$ is zero.

Let us now describe the moduli space $\MI_X(k)$ of instanton bundles
in this case in terms of nets of quadrics (in fact nets of conics).
Given an integer $k\ge 2$, we fix a vector space $I$ of dimension
$k$, and we consider the space $I \ts U$, and the space
$\bigwedge^2(I\ts U)$ of skew-symmetric 
vectors on it.
An element $\omega$ of this space can be regarded as a skew-symmetric
map $I^* \ts U^* \to I \ts U$, so for all integers $j$, $\bigwedge^2(I\ts U)$
contains the locally closed subvarieties of maps of given (necessarily even) rank:
\[
R_{j} = \{\omega \in \bigwedge^2(I\ts U)  \, | \, \rk(\omega) = 2j\}.
\]
We consider the space of nets of quadrics $\s^2I \ts B$ as a subspace of $\bigwedge^2(I\ts U)=\s^2 I
\ts \wedge^2 B \oplus \wedge^2 I \ts \s^2 U$, via the embedding $B
\mono \wedge^2 U$.
We have thus the varieties of nets of quadrics of fixed rank:
\[M_j = R_j \cap \s^2I \ts B.\]
We will also need to consider the space $\Delta$ of degenerate vectors
$\omega$, namely:
\[
\Delta = \{\omega \in \bigwedge^2(I\ts U)  \mid \exists \Lambda=\C^2
\subset U^* \, \mbox{with $\sigma_\Lambda = 0$ and $\omega|_{\Lambda \ts I^*}$ is not injective}\}.
\]
The group $\GL(k)=\GL(I)$ acts on $I$, hence on the nets of quadrics
$\s^2I \ts B$.
Also, $\GL(k)$ acts on $M_j$, for all $j$.

\begin{THM} \label{V5}
  There is a surjective $\GL(k)/\{\pm 1\}$-fibration $M_{2k+1} \setminus \Delta \to
  \MI_X(k)$ which is a geometric quotient for the action of $\GL(k)$.
\end{THM}

\subsubsection{Derived category of Fano threefolds of degree 5}

We denote by $\sU$ the rank $2$ universal sub-bundle on $X$, defined by
restriction from $\bG(2,U^*)$.
 We have $\Hom_X(\sU,\OO_X) \simeq \HH^0(X,\sU^*) \simeq U$, and
 $c_1(\sU)=-1$,  $c_2(\sU)=2$, $\sU^*(-1)\simeq \sU$. 
We have a canonical
exact sequence:
\begin{equation}
  \label{universal-5}
0 \to \sU \to U^* \ts \OO_X \to \sQ \to 0,  
\end{equation}
where $\sQ$ is the universal rank-$3$ quotient bundle.
According to \cite{orlov:V5} (see also \cite{dani:v5}), a decomposition of the category $\Db(X)$ in this case
has the form:
\begin{align*}
 & \Db(X) = \big\langle \OO_{X}(-1), \sU , \sQ^*, \OO_{X} \big\rangle,
\intertext{and its dual collection is:}
 & \Db(X) = \big\langle \OO_{X}(-1), \sQ(-1), \sU, \OO_{X} \big\rangle.
\end{align*}
We can use mutations to obtain an exceptional collection having the
form that we need.
Right mutating $\sQ^*$ through $\OO_X$, we get $\sU^*$ and this gives the decomposition  
\eqref{Db} satisfying \eqref{E0}, namely:
\begin{align}
 \nonumber & 
 \Db(X) = \big\langle \OO_{X}(-1), \sU , \OO_{X},\sU^* \big\rangle,
\intertext{and the dual collection:}
 \label{col-4-V5} &
 \Db(X) = \big\langle \OO_{X}(-1), \sQ(-1), \sR, \sU \big\rangle,
\end{align}
where $\sR$ is the kernel of the natural evaluation $U \ts \sU \to \OO_X$.
The following vanishing results are well-known:
\[
\HH^i(X,\sU(t))=0, \qquad \mbox{for all $t$ if $i=1,2$, and for $t\le
  0$ if $i=0$.}
\]

\subsubsection{Monads for instanton bundles on Del Pezzo threefolds of degree 5}

We will write the monad associated with an instanton bundle on a
Del Pezzo threefold $X$ of degree $5$.
Fix an integer $k\geq 2$, a vector space $I$ of dimension $k$, a
vector space $W$ of dimension $4k+2$ and a skew-symmetric isomorphism $D : W \to W^*$.
We look at an element $A$ of $I \ts W \ts U$ as a map:
\[
A : W^* \ts \OO_X \to I  \ts \sU^*,
\]
and we can consider:
\[
D \, A^\tra : I^{*} \ts \sU \to  W^* \ts \OO_X.
\]
Here the variety $\sQ_k=\sQ_{X,k}$ is given by $\sQ_k=\{A \in I \ts W
\ts U \mid A D A^\tra = 0\}$.

\begin{prop} \label{monad-5}
  Let $E$ be a sheaf in $\MI_X(k)$ with $k \ge 2$.
  Then $E$ is the cohomology of a monad:
  \[
  I^* \ts \sU \xr{D \, A^\tra} W^* \ts \OO_{X} \xr{A} I\ts \sU^{*},
  \]
  with $A,I,W,D$ as above.
  Conversely, the cohomology of such a monad is a $k$-instanton bundle on
  $X$.
\end{prop}

\begin{proof}
  We have to write the cohomology table \eqref{table} for a
  $k$-instanton $E$ with respect to the collection \eqref{col-4-V5},
  in order to write the complex \eqref{complex} whose cohomology is $E$.
  We first tensor \eqref{universal-5} by $E(-1)$, and use the instanton
  condition \eqref{instanton-i2} together with Lemma \ref{nulli} to
  obtain natural isomorphisms:
  \[
  \HH^{i-1}(X,E \ts \sQ(-1)) \simeq \HH^{i}(X,E \ts \sU(-1)) \simeq \HH^{3-i}(X,E \ts \sU)^*,
  \]
  where the second isomorphism uses Serre duality and the canonical
  isomorphisms $E \simeq E^*$, $\sU^*(-1) \simeq \sU$, $\omega_X \simeq \OO_X(-2)$.
  Now, by stability of $E$, $\sQ$ and $\sU$ (proved in
  \cite{dani:v5}), we get that the above groups vanish for $i=0,1,3$.
  We are only left with the space $\HH^1(X,E \ts \sQ(-1))$, which we
  call $I^*$, and whose dimension can be computed by Hirzebruch-Riemann-Roch and
  equals $k$. This takes care of the second column of \eqref{table},
  and also of the fourth, where the only non-vanishing group is
  $\HH^1(X,E \ts \sU) \simeq I$.

  We have to compute the cohomology groups $\HH^i(X,E \ts \sR)$. 
  By Lemma \ref{nulli} we have $\HH^{i-1}(X,E)=\HH^{i}(X,E(-2))=0$ for $i\ne 2$.
  We have the defining exact sequence:
  \begin{align}
    \label{prima} & 0 \to \sR \to U \ts \sU \to \OO_X \to 0, \\
\intertext{and the resolution of $\sR$, which is a part of a helix on $X$ (see again \cite{dani:v5}):}
    \label{seconda}     & 0 \to \OO_X(-2) \to V^* \ts \OO_X(-1) \to U^* \ts \sQ(-1) \to \sR
    \to 0.
  \end{align}
  Twisting \eqref{prima} by $E$ and using the vanishing we already
  proved, we get $\HH^i(X,E \ts \sR)=0$ for $i=0,3$.
  Twisting \eqref{seconda} by $E$ and again using the vanishing we already
  have, we also get $\HH^2(X,E \ts \sR)=0$.
  Thus we are only left with $\HH^1(X,E \ts \sR)$, we call this space
  $W^*$ and we compute its dimension $4k+2$ by Hirzebruch-Riemann-Roch.

  We have thus proved that a $k$-instanton $E$ is the cohomology of a
  monad of the form \eqref{monad-general}. We still have to check that
  the monad is self-dual. But this is rather clear, indeed
  the skew-symmetric duality $\kappa : E \to E^*$
  lifts to an isomorphism between our resolution and its transpose, and induces a
  skew-symmetric duality $D : W \to W^*$ in the middle term.

  Finally, to check the converse implication, we first note that the
  Chern classes of the cohomology sheaf $E$ of a monad of this form are indeed
  those of a $k$-instanton on $X$.
  In order to see that $E$ is stable, we need to check that
  $\HH^0(X,E) = 0$.
  Note that this space is isomorphic to $\Hom_X(E,\OO_X)$ because $A$
  is surjective so that $E$ is locally free.
  Then, a non-zero global section of $E$ gives a non-zero map $s:E \to
  \OO_X$, which therefore lifts to a map $W^* \ts \OO_X \to \OO_X$. 
  In particular $s$ is surjective.
  Note that $\ker(s)$ is reflexive of rank $1$ with vanishing $c_1$,
  so $\ker(s) \simeq \OO_X$.
  Hence $E$ is an extension of $\OO_X$ by
  $\OO_X$, i.e. $E \simeq \OO_X^2$, which contradicts our Chern class computation.
  
Moreover, twisting the monad by $\OO_X(-1)$, and using the
  cohomology vanishing for $\sU$, we immediately see that
  $E$ satisfies the instanton condition \eqref{instanton-i2}.
\end{proof}

\begin{proof}[Proof of Theorem \ref{V5}]
  Let us consider an element $\omega$ of $\bigwedge^2(I\ts U)$ as a
  skew-symmetric map $I^* \ts U^* \to I\ts U$.
  This induces a commutative diagram:
  \begin{equation}
    \label{4}
  \xymatrix{
    I^* \ts U^* \ar^-{\omega}[r] \ar^-{A_\omega^\tra}[d]& I \ts U  \\
    W \ar^-{D}[r] & W^*, \ar^-{A_\omega}[u]
  }
  \end{equation}
  where $D$ is the skew-symmetric duality induced by $\omega$ on its image $W$,
  and $A_\omega$ is the restriction of $\omega$ to its image.
  Since $\omega$ lies in $R_{2k+1}$ we have $\dim(W)=4k+2$.

  We can now consider $A_\omega$ as a map:
  \[
  A_\omega : W^* \ts \OO_X \to I  \ts \sU^*,
  \]
  and we would like to check that $A_\omega$ satisfies $A_\omega  \, D \, A_\omega^\tra=0$.
  This composition is a skew-symmetric map $I^*\ts \sU \to I \ts \sU^*$,
  that corresponds to a skew-symmetric map $I^* \ts U^* \to I \ts
  U$, which is nothing but $\omega$ itself.

  But the space of skew-symmetric maps $I^*\ts \sU \to I \ts \sU^*$ is:
  \begin{align*}
  \HH^0(X,\bigwedge^2(I \ts \sU^*)) & \simeq
  \wedge^2 I \ts \HH^0(X,\s^2\sU^*) \oplus
  \s^2 I \ts \HH^0(X,\wedge^2 \sU^*) \\
  & \simeq \wedge^2 I \ts \s^2U \oplus
  \s^2 I \ts V.
  \end{align*}
  So, $A_\omega \, D \, A_\omega^\tra$ as an element of
  $\bigwedge^2(I \ts U)$ is obtained by taking $\omega \in \bigwedge^2(I \ts U)$ and
  projecting it on the summand $\wedge^2 I \ts \s^2U \oplus \s^2 I \ts V$.
  But $\omega$ lies in $\s^2 I \ts B$, and $V = \wedge^2 U/B$, so $A_\omega \, D \, A_\omega^\tra=0$.
  Further, the fact that $A$ is a surjective map of sheaves 
  amounts to the fact that $A^\tra$ is an injective map of vector bundles, which is equivalent
  to the non-degeneracy condition for $\omega$ to lie away from $\Delta$.
  So,  $A_\omega$ defines a $k$-instanton bundle according to Proposition 
  \ref{monad-5}. This defines our map $M_{2k+1} \setminus \Delta \to
  \MI_X(k)$, that we denote by $\Psi$, and we let $E_\omega$ be the $k$-instanton bundle
  associated with $\omega$.

  To continue the proof, we show that $\Psi$ is surjective.
  This is essentially Proposition \ref{monad-5}.
  Indeed, any $k$-instanton $E$ is the
  cohomology of a monad given by a map $A$ and a duality $D$, and the
  pair $(A,D)$ gives an element $\omega$ in view of a diagram of the
  form \eqref{4}.
  The vector $\omega$ obtained in this way lies in $R_{2k+1}$ because
  $\dim(W)=4k+2$, and away from $\Delta$, because $A$ is surjective.
  Finally, we have checked that the condition $A_\omega\, D \, A_\omega^\tra=0$
  amounts to ask that $\omega$ lies in $\s^2 I \ts B \subset \bigwedge^2(I \ts U)$.

  In order to analyze the fibres of $\Psi$, we note that $E_\omega$ is equipped with a
  distinguished skew-symmetric duality $\kappa_\omega$, arising from the skew-symmetry of
  the defining monad.
  Moreover, it is easy to see that, given $\omega_1,\omega_2$, 
  there is an isomorphism $(E_{\omega_1},\kappa_{\omega_1}) \to (E_{\omega_2},\kappa_{\omega_2})$
  if and only if $\omega_1$ and $\omega_2$ are in the same $\GL(k)=\GL(I)$-orbit.
  Since $E$ is simple, we have $\omega_i = \pm \id_E$, so the stabilizer of $\omega$ under the $\GL(k)$-action is
  the set of automorphisms of $(E_\omega,\kappa_\omega)$, i.e. $\pm 1$.
  Then the map $\Psi$ has fibers $\GL(k)/\{\pm 1\}$.
  Moreover, according to \cite[Proposition 4.9]{kuznetsov:instanton},
  a net of quadric $\omega$ corresponding to an instanton bundle is
  semistable in the sense of GIT for the action of $\GL(k)$.
  The fact that the quotient is geometric is now clear.
\end{proof}

\begin{rmk}
  I do not know if the moduli space $\MI_X(k)$ affine, if $X$ is a Del Pezzo
  threefold of degree $5$. I do not know if it is smooth and/or
  irreducible for $k \ge 3$ (for $k=2$ it is, as proved in
  \cite{dani:v5}).
  I do not know if there are instantons with an
  $\mathrm{SL}_2$-structure (cf. \cite{faenzi:SL2-instanton} for the case of $\p^3$).
\end{rmk}

\subsection{Instanton bundles on Fano threefolds of degree 4}

Here we let $X$ be a Fano threefold $X$ of Picard number $1$, index
$2$ and degree $4$, so that, if $H_X$ is the (very) ample line bundle
generating $\Pic(X)$, we have $K_X=-2H_X$.
It is well-known that the Del Pezzo threefold $X$ is the complete intersection of two
quadrics in $\p^5$.
We will first recall the structure of the derived category of $X$ and
then state and prove a theorem on the moduli space of instanton
bundles on $X$, where monads are replaced by vector bundles on a curve
of genus $2$.

\subsubsection{Derived category of Fano threefolds of degree 4}

Given the variety $X$, we consider the
 moduli space $\Mo_X(2,1,2)$, which is isomorphic to a smooth curve
$\Gamma$ of genus $2$.
This curve is obtained as a double cover of the projective line representing
the pencil of quadrics vanishing on $X$, ramified along the $6$ points
corresponding to degenerate quadrics.
The moduli space is fine and we denote by $\sE$ a universal sheaf
on $X \times \Gamma$, determined up to a twist by a line bundle on
$\Gamma$. We denote by $\sE_y$ the sheaf in $\Mo_X(2,1,2)$ corresponding
to $y$.
We have:
\[
c_1(\sE) = H_X + N, \qquad c_2(\sE)=2L_X+H_X M+\eta,
\]
where $L_X$ is the class of a line in $X$, $M$ and $N$ are divisor
classes on $\Gamma$ of degree 
respectively $m$ and $2m-1$ (here $m$ is
an integer that we may choose arbitrarily), and $\eta$ lies in
$\HH^{1,2}(X)\ts \HH^{0,1}(\Gamma)$ and satisfies $\eta^2=4$.

It is not difficult to show that, if $\sE_y$ is a sheaf in
$\Mo_X(2,1,2)$, then $\sE_y$ is a stable, locally free, globally
generated sheaf, and that we have:
\begin{equation}
  \label{universal-4}
  0 \to \sE_{y'}(-1) \to \HH^0(X, \sE_y) \ts \OO_X \to \sE_y \to 0,
\end{equation}
where $y'$ is conjugate to $y$ in the $2:1$ cover $\Gamma \to \p^1$.

We denote by $p$ and $q$ the projections of $X \times \Gamma$ onto $X$
and $\Gamma$, and we consider the functors:
\begin{align*}
& \ph :  \Db(\Gamma) \to \Db(X), && \ph(a) = \RR p_*(q^*(a)
\ts \EE), \\
& \phx : \Db(X) \to \Db(\Gamma), && \phx(b) = \RR q_*(p^*(b)
\ts \EE^* \ts q^*(\omega_{\Gamma}))[1], \\
& \phs : \Db(X) \to \Db(\Gamma), && \phs(b) = \RR q_*(p^*(b)
\ts \EE^* \ts p^*(\omega_X))[3].
\end{align*}
We recall that $\ph$ is fully faithful, $\phs$ is left
adjoint to $\ph$, and $\phx$ is right adjoint to $\ph$.
If follows from \cite{bondal-orlov:semiorthogonal-arxiv} that we have the
following semiorthogonal decomposition:
\[
 \Db(X) \simeq \langle \OO_X(-1), \OO_X, \ph(\Db(\Gamma)) \rangle.
\]

An easy computation shows that $\cV=(\phs(\OO_X))^*$ is a vector bundle of rank
$4$ and degree $4(1-m)$ on $\Gamma$.

\subsubsection{Instanton bundles on Del Pezzo threefolds of degree 4
and Brill-Noether loci on curves of genus 2}

The goal of this section will be to prove the following result,
relating $k$-instanton bundles on $X$ to simple vector bundles of higher rank on
$\Gamma$, sitting in a specific Brill-Noether locus.
This locus is given by bundles $\sF$
 on $\Gamma$ having rank $k$ and degree $k(2-m)$ that have at least
 $k-2$ independent global sections when twisted by $\cV=(\phs(\OO_X))^*$.

\begin{THM} \label{dP4}
The map $E \to \phx(E)$ gives:
\begin{enumerate}[i)]
\item \label{prima-parte} an open immersion of $\MI_X(2)$ into the moduli space of
  semistable bundles of rank $2$ on $\Gamma$, so that $\MI_X(2)$ is a
  smooth irreducible fivefold;
\item \label{seconda-parte} for $k\ge 3$, an open immersion of $\MI_X(k)$ into the locus $\mathrm{W}(k)$
  defined as the isomorphisms classes of:
  \[
  \left\{
  \left.
    \begin{array}[h]{cc}
    \mbox{simple bundles $\sF$ on $\Gamma$} \\ 
    \mbox{of rank $k$ and degree $k(2-m)$}
    \end{array} \right| 
    \mbox{$\hh^0(\Gamma,\cV \ts \sF)=k-2$} \right\}
  \]
\end{enumerate}
\end{THM}

The main tool will be the next proposition, that provides the analogue
of the notion on monad for a classical instanton.

\begin{prop} \label{annullamenti}
  Let $E$ be a sheaf in $\MI_X(k)$, and set $\sF = \phx(E)$.
  Then:
  \begin{enumerate}[i)]
  \item \label{i} $\sF$ is a simple vector bundle of rank $k$
    and degree $k(2-m)$ on $\Gamma$;
  \item \label{ii} there is a functorial exact sequence:
    \begin{equation}
      \label{pseudo-monad}
      0 \to \OO_X^{k-2} \to \ph(\sF) \to E \to 0, 
    \end{equation}
    where $\ph(\sF)$ is a simple vector bundle of rank $k$ on $X$;
  \item \label{iii} we have:
  \[
  \hh^0(\Gamma,\cV \ts \sF)=k-2.
  \]
  \end{enumerate}
\end{prop}

\begin{proof}
  Since the space $\MI_X(k)$ is not empty we must have $k\ge 2$, see
  Lemma \ref{empty}.
  Let $y$ be a point of $\Gamma$.
  We first prove:
  \begin{equation}
    \label{solo1}
    \HH^i(X,E \ts \sE_y^*)=0, \qquad \mbox {for $i\ne 1$}.
  \end{equation}
  Notice that \eqref{solo1} holds for $i=0$ and $i=3$, as is easily
  proved using stability of $E$ and $\sE_y$ and Serre duality.
  Then we only need to prove that \eqref{solo1} holds for $i=2$.
  We use $\sE_y^* \simeq \sE_y(-1)$, we twist \eqref{universal-4} by $E(-1)$ and we use $\HH^i(X,E(-1))$
  for all $i$ (see Lemma \ref{nulli}).
  We get that  \eqref{solo1} holds for $i=2$ if $\HH^3(X,E\ts
  \sE_{y'}(-2))=0$.
  But this vanishing is clear from the stability of $E$ and $\sE_{y'}$
  and from Serre duality.
  
  We have thus proved \eqref{solo1}, and we note that $\sF$ is then a vector bundle on $\Gamma$, whose
  rank equals $\hh^1(X,E \ts \sE_y^*)=k$, which 
  can be computed by Hirzebruch-Riemann-Roch.
  The degree of $\sF$ can by computed by Grothendieck-Riemann-Roch.
  This finishes the proof of \eqref{i}, except the statement that
  $\sF$ is simple.

  To show \eqref{ii}, we note that $E$ lies in ${}^\perp \langle
  \OO_X(-1) \rangle$ by Lemma \ref{nulli}.
  By the same lemma, and by stability of $E$, we have $\HH^i(X,E)=0$ for $i\ne 1$ and by
  Hirzebruch-Riemann-Roch we have $\hh^1(X,E)=k-2$.
  Note also that $\Ext^1_X(E,\OO_X) \simeq \HH^1(X,E)$.
  Then we have a functorial exact sequence:
  \[
      0 \to \OO_X^{k-2} \to \tilde{E} \to E \to 0, 
  \]
  for some universal extension represented by a vector bundle $\tilde{E}$.
  Then, $\tilde{E}$ belongs to ${}^\perp \langle
  \OO_X(-1),\OO_X \rangle = \ph(\Db(\Gamma))$, so $\tilde{E} \simeq
  \ph(a)$ for some object $a$ in $\Db(\Gamma)$.
  But applying $\phx$ to the above sequence one gets $\sF \simeq
  \phx(\ph(a))$ so $\sF \simeq a$, and we get the resolution \eqref{pseudo-monad}.
  
  Note that, to finish the proof of \eqref{i} and \eqref{ii}, it will
  be enough to show  that $\sF$ is simple, because
  $\ph$ is fully faithful, so that $\ph(\sF)$ will be simple as well.
  In order to achieve this, we note that:
  \[
  \Hom_\Gamma(\sF,\sF) \simeq \Hom_X(\ph(\sF),E).
  \]
  Then, applying $\Hom_X(-,E)$ to \eqref{pseudo-monad} we get that
  this group has dimension $1$ since $E$ is stable (hence simple) and $\HH^0(X,E)=0$.

  It remains to show \eqref{iii}.
  Note that $\hh^0(\Gamma,\cV\ts \sF)$ equals the dimension of:
  \[
  \Hom_{\Gamma}(\cV^*,\sF) =  \Hom_{\Gamma}(\phs(\OO_X),\sF) \simeq \Hom_{X}(\OO_X,\ph(\sF)).
  \]
  Note also that, applying $\Hom_X(\OO_X,-)$ to \eqref{pseudo-monad} we get
  $\hh^0(X,\ph(\sF))=k-2$ since $\HH^0(X,E)=0$.
  This gives \eqref{iii}.
\end{proof}

\begin{proof}[Proof of Theorem \ref{dP4}]
  We first note that the map $E \to \phx(E)$ is injective.
  Indeed, after setting $\sF=\phx(E)$, we recall from the proof of Proposition  \ref{annullamenti}
  that $\HH^0(X,\ph(\sF)) \simeq \HH^0(\Gamma,\sF)$ and that this space
  has dimension $k-2$.
  Then, the map $\OO_X^{k-2} \to \ph(\sF)$ appearing in
  \eqref{pseudo-monad} is naturally determined from $\ph(\sF)$ as the natural evaluation of
  sections, so that we can recover its cokernel $E$ only from
  $\ph(\sF)$, i.e. from $\sF$.

  Then, we note that the differential of this map at the point of
  $\MI_X(k)$ given by $E$ identifies the tangent space of $\MI_X(k)$ and the
  space containing obstructions to $\MI_X(k)$ at $E$ with those of $\mathrm{W}(k)$
  at $\sF$.
  Indeed, we first rewrite \eqref{pseudo-monad} functorially as:
  \[
  0 \to \HH^1(X,E)^* \ts \OO_X \to \ph(\sF) \to E \to 0,
  \]
  and we deduce the natural isomorphisms:
  \begin{align*}
    & \HH^1(X,E)^* \simeq \HH^0(X,\ph(\sF)) \simeq
    \HH^0(\Gamma,\cV \ts \sF), \\
     & \HH^1(X,E) \simeq \HH^1(X,\ph(\sF)) \simeq 
    \HH^1(\Gamma,\cV \ts \sF).
  \end{align*}
  Using these isomorphisms, the adjoint pair $(\ph,\phx)$, and
  applying $\Hom_X(-,E)$ to the exact sequence above we get a 
  commutative diagram:
  \[
  \xymatrix@C-3.6ex{
    \Ext^1_X(E,E)  \ar[r] &  \Ext^1_X(\ph(\sF),E)  \ar[r] \ar@{=}[d]&  \HH^1(X,E) \ts \HH^1(X,E) \ar[r] \ar@{=}[d]&  \Ext^2_X(E,E)\\
    & \Ext^1_\Gamma(\sF,\sF) \ar[r] & \HH^0(\Gamma,\cV \ts \sF)^* \ts \HH^1(\Gamma,\cV \ts \sF)
  }
  \]
  This identifies $\Ext^1_X(E,E)$ and $\Ext^2_X(E,E)$ respectively
  with the kernel and the cokernel of the bottom map in this diagram,
  the dual Petri map. In turn, by standard Brill-Noether theory, these are the tangent space at $\sF$ of
  $\mathrm{W}(k)$ and the space of
  obstructions at $\sF$ of $\mathrm{W}(k)$.
  We have thus proved part \eqref{seconda-parte} of the theorem.

  In order to prove part \eqref{prima-parte}, we only need to show that
  $\sF$ is semistable, and to recall the well-known fact that the
  moduli space of semistable bundles of rank $2$ on $\Gamma$ is a smooth
  irreducible variety of dimension $5$.
  For the semistability of $\sF$, simply note that, by Proposition
  \ref{annullamenti}, we have $\HH^i(\Gamma,\cV \ts \sF)=0$ for
  $i=0,1$.
  This is a well-known sufficient condition for semistability.
\end{proof}


\section{Odd instanton bundles on prime Fano threefolds}
\label{i1}

Here we study the case of prime Fano threefolds, namely the case when
the Fano threefold $X$ has $i_X=1$.
In this case, the genus $g$ of $X$ is defined as the genus of the
canonical curve obtained by taking the intersection of two general
hyperplane sections of $X$, and one has $2g-2=H_X^3$.
According to Iskovskikh's classification \cite{iskovskih:I, iskovskih:II}
there are $10$ deformation classes of these threefolds, one for each
genus $g \in \{2,3,\ldots,12\} \setminus \{11\}$.

An (odd) $k$-instanton on $X$ is a rank-$2$ stable bundle $E$ with
$c_1(E)=-1$ and $c_2(E)=k$.
The Hilbert polynomial of $E$ is:
\[
\chi(E(t))=\frac{t}{3}((2g-2)t^2-3k+g+11).
\]

In \cite{brafa1:arxiv}, we proved an existence result, when $X$
is ordinary and not hyperelliptic
(see the definition in the introduction, and see
\cite{brafa1:arxiv} for a short review of these notions).
The claim is that, at least under these assumptions,
and setting $m_g=\lceil g/2 \rceil +1$,
there is a $(2k-g-2)$-dimensional,
generically smooth irreducible component of $\MI_X(k)$,
as soon as $k \ge m_g$, while
$\Mo_X(2,-1,k)$ is empty for $k < m_g$.

\medskip

Here, we will focus on the case $\HH^3(X)=0$
(recall that some cases with $\HH^3(X) \ne 0$ were studied in 
\cite{brafa1:arxiv} and \cite{brafa3:doi}).
This vanishing takes place if and only if $g=12$.
So, we let $k \ge 7$, and
there exists a generically smooth $(2k-14)$-dimensional
irreducible component of $\MI_{X}(k)$.
In this case, $X$ is not hyperelliptic, and $X$ is
ordinary unless $X$ is the Mukai-Umemura threefold, \cite{mukai-umemura,prokhorov:exotic}.

The threefold $X$ naturally sits in $\bG(3,7)$ and
in $\bG(2,8)$, and we denote by $\sU$ and $\sE$ the universal sub-bundles
of rank $3$ and of rank $2$, restricted to $X$.
We set $U=\Hom_X(\EE,\sU)$ (this vector space $U$ has dimension $4$).
It turns out that $\Db(X)$ admits a decomposition of the form
\eqref{Db} satisfying \eqref{E0}, with $\EE_2=\EE$ and $\EE_3=\sU$.
In fact $\sE$ is the only element in $\Mo_X(2,-1,7)=\MI_X(7)$.
The embedding of $X$ ins $\bG(2,8)$ comes from the description of $X$
as twisted cubics in $\p(U^*)$ that annihilate a net of quadrics
$\s^2 U^* \to B=\C^3$. The transpose of this map is
the structural embedding $B^*\mono \s^2 U$.
In this sense, a point $x$ of $X$ corresponds to a twisted cubic $T_x$
in $\p(U^*)$ equipped with the resolution:
\[
0 \to \OO_{\p(U^*)}(-4) \to \sU^*(-1)_x \ts \OO_{\p(U^*)}(-2) \xr{\psi_x} \sE_x \ts
\OO_{\p(U^*)}(-1) \to \omega_{T_x} \to 0.
\]

\medskip

Let us now describe $\MI_X(k+7)$ for $k \ge 1$ in terms of nets of
quadrics (in fact, of conics).
Take a vector space $I$ of dimension
$k$, and consider the spaces $I \ts U$ and
the symmetric square $\s^2(I\ts U)$.
An element $\alpha$ of $\s^2(I\ts U)$ is a symmetric
map $I^* \ts U^* \to I \ts U$, so for any integer $j$, $\s^2(I\ts U)$
contains the locally closed subvariety:
\[
R_{j} = \{\alpha \in \s^2(I\ts U) \mid \rk(\alpha) = j\}.
\]

The space of nets of quadrics $\s^2I \ts B^*$ is a subspace of $\s^2(I\ts U)=\s^2 I
\ts \s^2 U \oplus \wedge^2 I \ts \wedge^2 U$ via the structural embedding $B^*\mono \s^2 U$.
The varieties of nets of quadrics of fixed rank are:
\[M_j = R_j \cap (\s^2I \ts B^*).\]

Using the map $\psi_x$, we associate with a point $x \in X$ and a
point $0 \ne u \in U$ a linear map $\psi_x(u): \sU^*(-1)_x \to \sE_x$.
Therefore, for any $x \in X$ and $\alpha \in \s^2(I\ts U)$, we get a
linear map:
\[
\Psi_x(\alpha): I^* \ts \sU^*(-1)_x \to I \ts U \ts \sE_x.
\]
The variety $\Delta$ of degenerate elements of $M_j$ is thus:
\[
\Delta = \left\{\alpha \in M_j  \mid \exists   \mbox{$T_x$ twisted
    cubic that annihilates $B$ and $\Psi_x(\alpha)$ is not injective}\right\}.
\]
The group $\GL(k)=\GL(I)$ acts on $I$, on $M_j$, and on $M_j \setminus
\Delta$.

\begin{THM} \label{X12}
  There is a surjective $\GL(k)/\{\pm 1\}$-fibration $M_{3k+1} \setminus \Delta \to
  \MI_X(k+7)$ which is a geometric quotient for the action of
  $\GL(k)$.
\end{THM}

We will give some more results on $\MI_X(k)$, including a
monad-theoretic parametrization, a complete description for low $k$,
and an analysis in terms of jumping conics.

\subsection{Geometry of prime Fano threefolds of genus 12}

Here we briefly recall some of the geometry of a prime Fano threefold $X$
of genus $12$, the main references being \cite{mukai:fano-3-folds,schreyer:V22,faenzi:v22}.

\subsubsection{Constructions of threefolds of genus 12 and
  related bundles}

To give the first construction of our threefold, 
we have to fix two vector spaces $V = \C^7$ and $B = \C^3$ and
a net of alternating forms $\sigma : B \to \wedge^2 V^*$.
Then $X$ is obtained as the locus in $\bG(3,V)$:
\[
X = \{\Lambda \subset V \,\,|\,\, \sigma(b)(u \wedge v)=0 
 \mbox{ for any $u,v \in \Lambda$, for any $b \in B$} \},
\]
where we assume that $\sigma$ is general enough so that $X$ is smooth
of the expected codimension $9$.
The threefold $X$ is said to be of type $V_{22}$, for its degree
$H_X^3$ is $22$ (this is the maximal possible degree of a smooth prime Fano threefold).
Since $X$ sits in $\bG(3,7)$, we have 
a universal sub-bundle (of rank $3$) and quotient bundle
(of rank $4$), and we denote their restrictions to $X$ respectively by
$\sU$ and $\sQ$.
The universal exact sequence reads:
\begin{align}
 \label{universal22} & 0 \to \sU \to V \ts \OO_X \to \sQ \to 0,
\end{align}
We have $\Hom_X(\sU,\sQ^*) \simeq B$, and an exact sequence of vector
bundles:
\[
0 \to \sK \to B \ts \sU  \to \sQ^* \to 0,
\]
where $\sK$ is defined by the sequence, and turns out to be an
exceptional bundle of rank $5$.

As already mentioned, the second description of $X$ is as the
subvariety of Hilbert scheme 
of twisted cubics in $\p^3=\p(U^*)$ consisting of cubics $T$ whose ideal is annihilated by a
fixed net of quadrics, again parametrized by $B$, i.e. all quadrics in
the ideal of $T$ should be in the kernel of $\s^2 U^* \to B$.

This gives rise to a universal rank-$3$ bundle (the $3$ generators of the
ideal of a twisted cubic), which turns out to be isomorphic to $\sU$,
and to a universal rank-$2$ bundle (the $2$ syzygies among the $3$ generators)
which we denote by $\sE$. We have $\Hom_X(\sE,\sU) \simeq U$.
By this description, over $\p(B)$ we have a determinantal quartic
curve given by the degenerate quadrics in the net.
It is called the apolar quartic of $X$.
It turns out that $\sE^*$ embeds $X$ in $\bG(2,8)$, and we have a
universal exact sequence:
\begin{equation}
 \label{universalE} 0 \to \sE \to \HH^0(X,\sE^*)^* \ts \OO_X \to \sF \to 0,  
\end{equation}
for some universal quotient bundle $\sF$ of rank $6$.

Let us summarize the cohomology vanishing satisfied by the bundles
under consideration:
\begin{align}
  \label{acm}
  & \HH^i(X,\sE(t))= \HH^i(X,\sU(t))=0, && \mbox{for all $t$ if $i=1,2$,}\\ 
  \nonumber & \HH^0(X,\sE)= \HH^0(X,\sU)=\HH^0(X,\sU^*(-1))=0.
\end{align}

\subsubsection{Derived category of prime Fano threefolds of genus 12}

According to \cite{faenzi:v22}, the category $\Db(X)$ admits a pair of
dual semiorthogonal decompositions:
\begin{align}
 \label{col-1-V22} &  
 \Db(X) = \big\langle \sE , \sK, \sU , \OO_X \big\rangle, \\
 \label{col-2-V22} & 
 \Db(X) = \big\langle \sE, \sU, \sQ^*, \OO_{X} \big\rangle.
\end{align}

Mutating the exceptional collections above, 
we get a semiorthogonal decomposition satisfying \eqref{E0}:
\begin{align}
 \nonumber & 
 \Db(X) = \big\langle \OO_{X}(-1), \sU^*(-1) , \sE, \sU \big\rangle,
\intertext{and the dual decomposition:}
 \label{col-4-V22} &
 \Db(X) = \big\langle \OO_{X}, \sQ, \sR^*, \sU^* \big\rangle,
\end{align}
where $\sR$ is defined as right mutation of $\sE$ with respect to $\sU$:
\begin{equation}
  \label{R}
  0 \to \sE \to U^* \ts \sU \to \sR \to 0.  
\end{equation}

\subsubsection{Conics contained in a prime Fano threefolds of genus 12}

One way to study odd instantons on a smooth prime Fano threefold of
genus is to look at the their scheme of jumping conics.
We will develop this in the case of threefolds of genus $12$.

In order to do so, we first take a look at the Hilbert
scheme $\sH^2_0(X)$ of conics contained in $X$.
It turns out that $\sH^2_0(X)$ is identified with the projective plane
$\p^2=\p(B^*)$, cf \cite[5.2.15]{fano-encyclo}.
However we will use the following description.

\begin{lem}
  There is a natural isomorphism $B \simeq \Hom_X(\sU,\sQ^*)$, and:
  \begin{enumerate}[i)]
  \item \label{ideal}
    any element $b \in \p(B^*)$ gives a conic $C_b$ contained in $X$ and an exact sequence:
    \[
    0 \to \sU \xr{b} \sQ^* \to \sI_{C_b} \to 0;
    \]
  \item \label{P2} the map $b \mapsto C_b$ defines an isomorphism $\p(B^*) \simeq
    \sH^{2}_0(X)$;
  \item \label{tre} over the product $X \times \p(B^*)$ there is an exact sequence:
    \begin{equation}
      \label{canconuni}
    0 \to \OO_{X \times \p(B^*)}(-1,-3) \to p_1^*(\sQ(-1)) \to \sU^*(-1)\boxtimes
    \OO_{\p(B^*)}(1) \to \omega_\sC \to 0, 
   \end{equation}
    where $p_1$ and $p_2$ denote the projections from $X \times
    \p(B^*)$ onto the two factors, and $\sC$ is the universal conic.
  \end{enumerate}
\end{lem}

\begin{proof}
  It is proved in \cite[Lemma 6.1]{faenzi:v22} that there is a natural
  isomorphism $B \simeq \Hom_X(\sU,\sQ^*)$.
  Given $0 \ne b \in B$, we have a map $b : \sU \to \sQ^*$, and using the
  stability of $\sU$ and $\sQ^*$ (proved in \cite[Lemma
  6.2]{faenzi:v22}) we get that $b$ is necessarily injective.
  Then $b^\tra : \sQ \to \sU^*$ has generic rank $3$ and
  $\ker(b^\tra)$ is reflexive of rank $1$, hence invertible, i.e.,
  $\ker(b^\tra) \simeq \OO_X(t)$, while $\coker(b^\tra)$ is a coherent
  torsion
  sheaf.
  By stability of $\sU$ we have $t\le 0$.
  But we have $c_1(\coker(b^\tra))=t \ge 0$ so $t=0$.
  Using the cohomology of $\sU^*(-1)$, $\sQ(-1)$ and $\OO_X(-1)$ one
  sees that $\HH^0(\coker(b^\tra(-1)))=0$.
  So, $\coker(b^\tra)$ is supported on
  a curve $C_b$ without embedded points, of degree $2$, and is locally
  free of rank $1$ on $C_b$. 
  Therefore $\coker(b)$ is the ideal sheaf of $C_b$, and by a Hilbert
  polynomial computation we see that $C_b$ has arithmetic genus $0$,
  i.e., $C_b$ is a conic.
  We have proved \eqref{ideal}.
  Dually, we get:
  \begin{equation}
    \label{cancon}
  0 \to \OO_X(-1) \to \sQ(-1) \xr{b^\tra} \sU^*(-1) \to \omega_{C_b} \to 0.    
  \end{equation}

  Associating $C_b$ with $b$ thus defines a map $\p(B^*) \to
  \sH^{2}_0(X)$, which is clearly injective.
  Since we already know $\sH^{2}_0(X) \simeq \p^2$, we conclude that
  $\p(B^*) \to \sH^{2}_0(X)$ is an isomorphism.
  However, one could construct the inverse of this map by taking the
  Beilinson resolution of the ideal sheaf $\sI_{C}$ of a conic $C$
  according to the decomposition \eqref{col-2-V22}.
  Anyway we have proved \eqref{P2}.

  Writing the middle map of \eqref{cancon} universally with respect to
  $b \in \p(B^*)$ gives the middle map of \eqref{canconuni}.
  Fixing a point $x\in X$, the associated map $\sQ(-1)_x \ts
  \OO_{\p(B^*)} \to \sU^*(-1)_x \ts \OO_{\p(B^*)}(1)$ is a $4 \times
  3$ matrix of linear forms, degenerating on a subscheme $Z$ of length $6$
  in $\p(B^*)$, consisting of the $6$ conics through $x$, see
  \cite{takeuchi:birational}.
  One sees easily that the kernel of this matrix is
  $\OO_{\p(B^*)}(-3)$ and that its cokernel is $\omega_Z$.
  Summing up, we have proved \eqref{tre}.
\end{proof}

Dualizing \eqref{canconuni} and twisting by $\omega_X \boxtimes
\omega_{\p(B^*)}$ we get:
    \begin{equation}
      \label{ouni}
    0 \to \sU\boxtimes \OO_{\p(B^*)}(-4) \to \sQ^* \boxtimes \OO_{\p(B^*)}(-3)\to  \OO_{X \times \p(B^*)} \to \OO_\sC \to 0.
   \end{equation}

\subsection{Monads for instantons on Fano threefolds of genus 12}

According to Definition \ref{instanton}, a coherent sheaf $E$ is an odd
$k$-instanton on $X$ if $E$ is a rank-$2$ bundle with:
\[
c_{1}(E)=-1, \qquad c_{2}(E)=k, \qquad \HH^{1}(X,E)=0.
\]
We have said that $\MI_X(k)$ is empty for $k\le 6$, 
and that it contains a unique rank-$2$ bundle $\sE$ if $k=7$
(we refer to \cite{kuznetsov:V22, brafa2}).

Increasing $k$ to $k=8$, we can control the whole moduli space
$\Mo_X(2,-1,8)$, as shown by the next proposition that puts together
some ingredients, most of them already available in the literature. 
Indeed, the sheaves in this space that are not locally free fail to be
so precisely along a line: a statement proved in \cite[Proposition 3.5]{brafa1:arxiv}.
Intuitively, this is due to the fact that the $c_2$ of the double dual must go
down by one because there are no rank-$2$ semistable bundles with
$c_1=-1$ and lower $c_2$.

\begin{prop} 
  The moduli space $\Mo_X(2,-1,8)$ is identified with $\p^2$, and
  its open piece of locally free sheaves is $\MI_X(8)$.
  The complement of this piece is the apolar quartic of $X$, and is in
  bijection with the set of lines contained in $X$.
\end{prop}

\begin{proof}
  Let $E$ be a sheaf in $\Mo_X(2,-1,8)$.
  By \cite[Proposition 3.5, part (i)]{brafa1:arxiv}, we know that $\HH^i(X,E)=0$ for all $i$
  and $\HH^i(X,E(1))=0$ for $i>0$. This shows that any locally free
  sheaf in $\Mo_X(2,-1,8)$ lies in $\MI_X(8)$.
  Further, by \cite[Proposition 3.5, part (iii)]{brafa1:arxiv} we know that either $E$ is locally free,
  or it fits into an exact sequence:
  \begin{equation}
    \label{non}
    0 \to E \to \sE \to \OO_L \to 0,    
  \end{equation}
  where $L$ is a line contained in $X$.
  
  We compute the complex associated with $E$ according to the 
  decomposition \eqref{col-1-V22} and we calculate thus the cohomology
  table of $E$ tensored with the bundles appearing in
  \eqref{col-2-V22}.
  We already settled this almost completely in the previous
  proposition.
  Indeed, assume first that $E$ is locally free.
  Then, we established that $\hh^i(X,E)=\hh^i(X,E\ts
  \sE)=0$ for all $i$ (indeed the only relevant value was $\hh^1(X,E\ts
  \sE)=k-8=0$).
  Further, we have $\hh^1(X,E\ts \sU)=\hh^1(X,E\ts \sQ^*)=1$, and all
  remaining cohomology groups are zero, so we get a resolution of $E$:
  \[
  0 \to E \to \sK \xr{f} \sU \to 0,
  \]
  and $f$ lies in $\Hom_X(\sK,\sQ)=B$, so we have a map $\MI_X(8) \to
  \p^2$ given by $E \mapsto [f]$.
  We proved in \cite{enrique-dani:v22} that a morphism $g:\sK \to
  \sU$ fails to be surjective
  if and only if $[g]$ lies in the apolar quartic of $X$, and that in
  this case we have an exact sequence:
  \begin{equation}
    \label{retta}
  0 \to \sE \to \sK \xr{g} \sU \to \OO_L(-1) \to 0.    
  \end{equation}

  On the other hand, if $g$ is surjective then its kernel clearly lies
  in $\MI_X(8)$, and $[f]$ and $E$ determine each other this way.
  This proves that $\MI_X(8)$ is identified with the complement of the
  apolar quartic in $\p^2$. Further, this quartic is in bijection with
  the set of lines contained in $X$, for any line $L \subset X$
  determines canonically a long exact sequence as above.

  To complete the proof, we only need to look at the class of sheaves $E$ which are
  not locally free and therefore fit into \eqref{non}.
  The sheaf $E$ is clearly determined by $L$, and still determines a
  map $f : \sK \to \sU$ that this time will belong to the apolar
  quartic.
  Conversely, given $g : \sK \to \sU$ in the apolar quartic, dualizing
  \eqref{retta} we see that $E$ is obtained as the cokernel:
  \[
  0 \to \sU^*(-1) \xr{g^\tra} \sK^*(-1) \to E \to 0.
  \]
  This gives the inverse map from the apolar quartic to the class of
  non-locally-free sheaves, and this clearly agrees with the
  construction given for locally free sheaves.
  The proof of the proposition is thus finished.
\end{proof}

For higher $k$, we give the following monad-theoretic description of $\MI_X(k)$.

\begin{prop} \label{v22-basic}
Let $k\ge 1$, and let $E$ be a sheaf in $\MI_X(k+7)$. Then $E$ is the cohomology of a monad of the form:
\[
I^* \ts \sU^{*}(-1) \xr{D \, A^\tra} W^* \ts \sE \xr{A} I\ts \sU,
\]
where $I \simeq \C^{k}$, and $W \simeq \C^{3k+1}$, and $D:W^{*} \to W$ is a symmetric duality.
Conversely, the cohomology of such a monad is a a sheaf in $\MI_X(k+7)$.
\end{prop}

\begin{proof}
  We first have to write the cohomology table \eqref{table} with respect to
  the collection \eqref{col-4-V22}.
  By Lemma \ref{nulli}, we get $\HH^i(X,E)=0$ for all $i$.
  Twisting \eqref{universal22} by $E$ we get:
  \[
  \HH^{i-1}(X,E\ts \sQ) \simeq \HH^{i}(X,E\ts \sU) \simeq \HH^{3-i}(X,E \ts \sU^*)^*,
  \]
  where the second isomorphism uses Serre duality and the fact $E^*(-1)
  \simeq E$.
  Stability of $\sU$ and $\sQ$ is proved in \cite{faenzi:v22}, and
  implies that the above groups vanish for $i=0,1,3$.
  Indeed, the slope of the bundle whose group of global sections is
  being computed is always strictly negative.
  Therefore we are left with $\HH^{1}(X,E\ts \sQ)$, we denote this
  vector space by $I^*$ and we compute its dimension $k$ by Hirzebruch-Riemann-Roch.
  We deduce $\HH^{1}(X,E \ts \sU^*) \simeq I$.
  This takes care of the second and of the fourth columns in our
  cohomology table.

  To finish the computation of cohomology, we have to calculate $\HH^{i}(X,E\ts \sR^*)$.
  To do this, we first compute $\HH^{i}(X,E\ts \sE^*)$.
  We note that, by stability, $\HH^i(X,E \ts \sE^*)$, vanishes for
  $i=0$, indeed any morphism $\sE \to E$ must be an isomorphism, but
  $c_2(\sE)=7 < c_2(E)$.
  The same group vanishes for $i=3$, in view of Serre duality.
  We recall from \cite{faenzi:v22} that the quotient bundle $\sF$ is stable,
  and we note that this implies $\HH^0(X,E \ts \sF)=0$ and in turn
  $\hh^2(X,E \ts \sE^*)=\hh^1(X,E \ts \sE)=0$ by using Serre duality
  and \eqref{universalE}.
  We are only left with $\HH^1(X,E \ts \sE^*)$ (its dimension
  equals $k-1$ by Hirzebruch-Riemann-Roch).

  We can now proceed to compute $\HH^{i}(X,E\ts \sR^*)$.
  We dualize \eqref{R}, tensor it by $E$, and take
  cohomology. By the vanishing already obtained, we get that this
  group vanishes for $i=0,3$.
  Now we use the following resolution of $\sR^*$, which is a part of a
  helix on $X$:
  \[
  0 \to \sE \to \HH^0(X,\sE^*)^* \ts \OO_X \to U^* \ts \sQ^* \to \sR^*  \to 0.
  \]
  Tensoring it by $E$, taking cohomology, and using the vanishing
  already obtained, we get $\HH^i(X,E\ts \sR^*)=0$ for $i=2$.
  This only leaves $\HH^1(X,E\ts \sR^*)$.
  We call this space $W^*$, and we compute $\dim(W)=3k+1$ by Hirzebruch-Riemann-Roch. 

  To continue the proof, we have to establish the self-duality of the
  monad. This is provided, as usual, by the skew-symmetric duality $\kappa : E
  \to E^*(-1)$.
  We lift $\kappa$ to an isomorphism between our monad and its dual,
  twisted by $\OO_X(-1)$.
  This provides a duality of $W^* \ts \sE$ that lies in:
  \[
  \HH^0(X,\bigwedge^2 (W^* \ts \sE^*) \ts \OO_X(-1)) \simeq 
  \left\{  \begin{array}[h]{c}
  \wedge^2  W^* \ts \HH^0(X,\s^2 \sE^*(-1)) \\    
  \oplus  \\
  \s^2 W^* \ts \HH^0(X,\wedge^2 \sE^*(-1)).
  \end{array} \right.
  \]
  In the direct sum above, the former term is zero, while the second
  one is naturally identified with $\s^2 W^*$.
  So we have a distinguished element of $\s^2 W^*$ that corresponds to
  the symmetric duality $D$, and the monad is self-dual.

  To finish the proof, we must ensure that the converse implication
  holds, so we let $E$ be the cohomology of a monad of the form under consideration.
  It is easy to see that $E$ is a bundle of rank $2$ with the
  appropriate Chern classes, and using \eqref{acm} we get that $E$
  fulfills $\HH^i(X,E)=0$ for all $i$.
  This suffices to guarantee that $E$ lies in $\MI_X(k+7)$ and we are done.
\end{proof}

We have now proved Theorem \ref{X12}.
This, together with Proposition \ref{monad-5} and Theorem \ref{quadric},
achieves the proof of Theorem \ref{generale}.

\begin{rmk}
  I do not know if the moduli space $\MI_X(k)$ is affine, smooth,
  irreducible when $X$ is a prime Fano
  threefold of genus $12$, and $k\ge 9$.
  I haven't found examples of bundles with an $\mathrm{SL}_2$-structure supported
  on the Mukai-Umemura threefold.
\end{rmk}

\subsection{Nets of quadrics for instantons on Fano threefolds of
  genus 12} 

This subsection is devoted to the proof of Theorem \ref{X12}, we refer
to the beginning of the section for notation.

First of all, we need to establish the natural isomorphisms:
\begin{equation}
  \label{natu}
  V \simeq \ker(\s^2 U^* \to B), \qquad \HH^0(X,\s^2 \sU^*(1)) \simeq \wedge^2 U.
\end{equation}
The first one is proved in \cite{schreyer:V22}, see also \cite[Section
4]{faenzi:v22}.
For the second one, we use \cite[Remark 8.1]{faenzi:v22} (and its proof).

To begin the proof of Theorem \ref{X12}, we show that $\alpha \in
M_{3k+1}\setminus \Delta$ gives a $(k+7)$-instanton.
To do this, observe that an element $\alpha
\in \s^2(I\ts U)$ induces a commutative diagram:
\[
    \xymatrix{
      I^* \ts U^* \ar^-{\alpha}[r] \ar^-{A_\alpha^\tra}[d]& I \ts U  \\
      W \ar^-{D}[r] & W^*, \ar^-{A_\alpha}[u]
    }
  \]
  where $D$ a symmetric duality induced by $\alpha$ on its image $W$.
  Since $\alpha$ lies in $R_{3k+1}$ we have $\dim(W)=3k+1$.
  Then, we have $A_\alpha : W^* \ts \sE \to I  \ts \sU$, and
  $A_\alpha \, D \, A_\alpha^\tra$ is a skew-symmetric map
  $I^*\ts \sU^*(-1) \to I \ts \sU$, so lies in:
  \begin{align*}
    \HH^0(X,\bigwedge^2(I \ts \sU)\ts \OO_X(1)) & \simeq
    \wedge^2 I \ts \HH^0(X,\s^2\sU(1)) \oplus
    \s^2 I \ts \HH^0(X,\wedge^2 \sU(1)) \\
    & \simeq \wedge^2 I \ts \wedge^2U \oplus
    \s^2 I \ts V^*,
  \end{align*}
  where we used \eqref{natu}.
  Therefore, $A_\omega \, D \, A_\omega^\tra$ is obtained by taking
  $\alpha \in \s^2(I \ts U)$ and 
  projecting it on the summand $\wedge^2 I \ts \wedge^2U \oplus \s^2 I
  \ts V^*$. 
  Since $\alpha$ lies in $\s^2 I \ts B^*$, and $V^* = \s^2 U/B^*$, we
  get $A_\alpha \, D \, A_\alpha^\tra=0$.
  Also, $A_\alpha$ is a surjective map of sheaves 
  if $\alpha$ does not lie in $\Delta$.
  The conclusion that $A_\alpha$ gives a $(k+7)$-instanton follows
  from Proposition \ref{v22-basic}.
  We have thus obtained a map $M_{3k+1}\setminus \Delta \to \MI_X(k)$
  which is invariant for $\GL(k)$.

  \medskip
  Next, we want to prove that this map is surjective. But this is
  clear from Proposition \ref{v22-basic}, cf. the  argument
  used in the proof of Theorem \ref{V5}.

  \medskip
  The final step to prove Theorem \ref{X12} is to show that our nets
  of quadrics are semistable for the $\GL(k)$-action.
  This will be easier once we have an alternative
  monad-theoretic description of $E$.
  To exhibit it, recall from \cite[Lemma 6.1]{faenzi:v22} that
  $\Hom_X(\sK,\sU) \simeq \Hom_X(\sU^*(-1),\sK^*(-1)) \simeq B^*$.
  Then, for all $\alpha \in \s^2I \ts B^*$, we look at the associated
  map $I^* \to I \ts B^*$ and set $B_\alpha$ for the
  following natural composition:
    \[
    I^* \ts \sU^*(-1) \to I \ts B^* \ts \sU^*(-1) \to I \ts \sK^*(-1).
    \]

  \begin{lem}
    Let $E$ be the $(k+7)$-instanton on $X$ associated with
    $\alpha$. Then $E$ is the cohomology of a monad:
    \[
    I^* \ts \sU^*(-1) \xr{B_\alpha} I \ts \sK^*(-1) \to C \ts \sE,
    \]
    where $C = \coker(W^* \to I \ts U)$ has dimension $k-1$.
  \end{lem}

  \begin{proof}
    Dualizing and twisting by $\OO_X(-1)$ the sequence
    \cite[(42)]{faenzi:v22}, we get:
    \[
    0 \to \sK^*(-1) \to U \ts \sE \to \sU \to 0.
    \]
    Tensoring this sequence with $I$ and fitting it with the monad of
    Proposition \ref{v22-basic} gives rise to the diagram:
    \[
    \xymatrix{
      \ker(A) \ar[d] \ar[r] & I \ts \sK^*(-1) \ar[d] \ar[r] & C \ts \sE \ar@{=}[d]\\
      W^* \ts \sE \ar[r] \ar[d]& I \ts U \ts \sE \ar[r] \ar[d]& C \ts \sE \\
      I \ts \sU \ar@{=}[r] &       I \ts \sU 
    }
    \]
    Since $E = \ker(A)/I^* \ts \sU^*(-1)$, the top line of the diagram
    shows the desired monad. The leftmost map is $B_\alpha$ by construction.
    The dimension of $C$ is clear.
  \end{proof}

  Observe that $\alpha \in R_{3k+1}$ lies away from $\Delta$ iff $B_\alpha$ is fibrewise injective.

\begin{lem}
  Let $\alpha \in \s^2 I \ts B^*$ be a net of quadrics with $\rk(\alpha) = 3k+1$ and
  $\alpha \not \in \Delta$. Then $\alpha$ is semistable for the
  $\GL(k)$-action.
\end{lem}

\begin{proof}
  The argument is very similar to \cite[Proposition
  4.9]{kuznetsov:instanton}, but the computation needs to be carried
  our explicitly.
  From \cite{wall:nets-of-quadrics}, we recall that $\alpha \in \s^2 I
  \ts B^*$ is not
  semistable iff there are subspaces $I_1,I_2$ of $I^*$ with
  $\dim(I_1)+\dim(I_2) > k$ such that $B \to \s^2 I \to I_1 \ts I_2$
  is zero.
  Exactly as in \cite[Proposition 4.9]{kuznetsov:instanton}, we get
  $\dim(I_1)>\dim(I_2)^\perp$ and a natural induced diagram:
  \[
  \xymatrix{
    0 \ar[r] & I_1 \ts \sU^*(-1) \ar[r] \ar^-{b}[d] &  I \ts \sU^*(-1)  \ar[r] \ar^-{B_\alpha}[d] &  I/I_1 \ts \sU^*(-1) \ar[r] \ar^-{c}[d] & 0\\
    0 \ar[r] & I_2^\perp \ts \sK^*(-1) \ar[r]  &  I^* \ts \sK^*(-1)  \ar[r]  &  I_2^* \ts \sK^*(-1) \ar[r]  & 0
  }
  \]
  This gives an exact sequence:
  \[
  0 \to \ker(c) \to \coker(b) \xr{d} C \ts \sE \to \coker(c) \to 0.
  \]
  
Now, using the values $c_1(K)=-2$, $\rk(K)=5$ (see \cite[Lemma
  6.1]{faenzi:v22}), and setting $i_1 = \dim(I_1)$, $i_2=\dim(I_2^\perp)$
  we get $c_1(\coker(b))=-3i_2+2i_1$ and $\rk(\coker(b))=5i_2-3i_1$.
    By stability of $\sE$ and $\sU$ (proved in \cite{faenzi:v22})
  we get:
  \begin{align*}
     c_1(\ker(c)) \le -\frac 23 \rk(\ker(c)), &&   c_1(\im(d)) \le-\frac 12 \rk(\im(d)).
  \end{align*}
  This gives:
  \begin{align*}
  c_1(\coker(b)) & \le c_1(\ker(c))- \frac 12 \rk(\coker(b)) + \frac  12 \rk(\ker(c)) \le\\
  & \le - \frac 23 \rk(\ker(c))- \frac 12 \rk(\coker(b)) + \frac  12 \rk(\ker(c)).
  \end{align*}
  We obtain:
  \[
  -6i_2+4i_1 \le -\frac 13 \rk(\ker(c)) - 5i_2+3i_1,
  \]
  which contradicts $i_1 > i_2$.
\end{proof}

\subsection{Jumping conics of instantons on Fano threefolds of
  genus 12}
 
Here we will develop some basic considerations to describe the set of
jumping conics of an odd instanton on a smooth prime Fano threefold
$X$ of genus $12$.

\begin{dfn}
Given an odd $k$-instanton $E$ over $X$, and a conic $C$ contained in
$X$, we say that $C$ is a {\it jumping conic} for $E$ if:
\[
\HH^1(C,E|_C) \ne 0.
\]
Looking at the universal conic $\sC$ as a variety dominating $X$ and
$\p(B^*)$, and denoting by $p$ and $q$ the projections,
we can then define the {\it scheme of jumping conics} as the support
of the coherent sheaf:
\[
\RR^1q_* (p^*(E) \ts \OO_\sC).
\]
We says that $E$ has {\it generic splitting} if not all conics
of $X$ are jumping for $E$.
\end{dfn}

\begin{lem}
  Let $E$ be a $k$-instanton on $X$. 
  Then the scheme of jumping conics of $E$ is supported on the
  cokernel a symmetric matrix of linear forms:
  \[
  M_E : I^* \ts \OO_{\p(B^*)}(-4) \to I \ts \OO_{\p(B^*)}(-3).
  \]
\end{lem}

\begin{proof}
  In order to compute $\RR^1q_* (p^*(E) \ts \OO_\sC)$, we use
  \eqref{ouni}.
  Since $E$ is a $k$-instanton, we have:
  \[
  \HH^k(X,E)=0, \qquad   \HH^j(X,E \ts \sQ^*)= \HH^j(X,E \ts \sU)=0, 
  \]
  for all $k$ (Lemma \ref{nulli}) and for $j\ne 2$, which can be seen from
  the proof of Proposition \ref{v22-basic} using that $\HH^j(X,E \ts \sQ^*)
  \simeq \HH^{3-j}(X,E^*(-1) \ts \sQ)^* \simeq \HH^{3-j}(X,E \ts
  \sQ)^*$ and $\HH^j(X,E \ts \sU)
  \simeq \HH^{3-j}(X,E^*(-1) \ts \sU^*)^* \simeq \HH^{3-j}(X,E \ts
  \sU^*)^*$. 
  By the same reason, we have natural isomorphisms:
  \[
  \HH^2(X,E \ts \sQ^*) \simeq I, \qquad \HH^2(X,E \ts \sU) \simeq I^*. 
  \]
  Therefore, tensoring \eqref{ouni} with $p^*(E)$ and taking direct
  images, we get an exact sequence:
  \[
  0 \to q_* p^*(E \ts \OO_\sC) \to I^* \ts \OO_{\p(B^*)}(-4) \to I \ts
  \OO_{\p(B^*)}(-3) \to \RR^1q_* (p^*(E) \ts \OO_\sC) \to 0.
  \]
\end{proof}

\begin{prop}
  The scheme of jumping conics of $(k+7)$-instanton $E$ with generic
  splitting is a plane curve $\sC$ of degree $k$, and is equipped with
  a coherent sheaf
  $\sL$ such that $\HHom_{\sC}(\sL,\OO_{\sC}) \simeq \sL(k-1)$.
  Moreover, $E$ is determined by $(\sC,\sL)$.
\end{prop}

\begin{proof}
  Let $E$ be a $(k+7)$-instanton on $X$.
  If $E$ has generic splitting, then there exists $b \in \p(B^*)$ such
  that the conic $C_b$ is not 
  jumping for $E$, so that map $M_E$ associated with $E$ according to the
  previous lemma is injective at $b$.
  Therefore, $\coker(M_E)$ is supported on a curve $\sC$ of degree
  $k$ in $\p(B^*)$, defined by the equation $\det(M_E)$.
  Let $i : \sC \mono \p(B^*)$ be the embedding.
  The sheaf $\coker(M_E) \ts \OO_{\p(B^*)}(3)$ is supported on $\sC$
  and we have:
  \begin{equation}
    \label{det}
  0 \to I^* \ts \OO_{\p(B^*)}(-1) \xr{M_E} I \ts \OO_{\p(B^*)} \to i_*(\sL)\to 0.
  \end{equation}
  Then $\sL$ is a Cohen-Macaulay (hence torsion-free) sheaf on $\sC$.
  Since $M_E$ is symmetric, dualizing the above sequence we find:
  \[
  i_*(\sL) \simeq \EExt^1_{\p(B^*)}(i_*(\sL),\OO_{\p(B^*)}(-1)) \simeq i_*(\HHom_{\sC}(\sL,\OO_{\sC}(k-1))),
  \]
  where the last isomorphism follows from Grothendieck duality.

  To check that $E$ is determined by $(\sC,\sL)$, note that the exact
  sequence \eqref{det} is the sheafification of the minimal graded
  free resolution of the module over $\C[x_0,x_1,x_2]$ associated with $i_*(\sL)$.
  As such, it is unique up to conjugation by $\GL(I)$.
  Therefore, from $(\sC,\sL)$ we reconstruct $M_E$, i.e., the net of
  quadrics as en element of $\s^2 I \ts B^*$, up to $\GL(I)$-action.
  This gives back $E$
  according to Theorem \ref{X12}.
\end{proof}

\section{Geometric quotients for moduli of instanton bundles}

\label{section-quotient}

Here, we will carry out some considerations to check that the
moduli space $\MI_X(k)$ is a geometric quotient, relying on the monad-theoretic presentation.
In other words, we show the last statement of Theorem \ref{generale}.

\begin{lem}
  The map $\sQ_{X,k}^\circ \to \MI_X(k)$, that sends $A$ to the
  cohomology sheaf $\ker(A)/\im(A \, D^\tra)$ is a geometric quotient for
  the $G_k$-action.
\end{lem}

\begin{proof}
  We will deduce our statement from some well-known facts concerning
  the action of a larger group $\hat G_k$ on a larger variety
  $\hat \sQ_{X,k}^\circ$.
  
  Let us first define $\hat \sQ_{X,k}$.
  Fix vector spaces $I$ and $W$ as in Table \eqref{valori}, and take
  another vector space $J$ with $\dim(J)=\dim(I)$.
  Consider the set of pairs $(A,B)$ in $(I \ts W \ts U) \times (J^* \ts W \ts U)$ as pairs of morphisms:
  \[
  A : W^* \ts \sE_2 \to I  \ts \sE_3, \qquad B : J \ts \sE_1 \to  W^* \ts \sE_2.
  \]
  We define the variety $\hat \sQ_{X,k}$ by:
  \[
  \hat \sQ_{X,k}= \{(A,B) \in (I \ts W \ts U) \times (J^* \ts W
  \ts U) \mid A \, B = 0\}.
  \]
  The relevant open piece $\hat \sQ_{X,k}^\circ$ of $\hat \sQ_{X,k}$
  is:
  \[
  \hat\sQ_{X,k}^\circ= \{(A,B) \in \hat \sQ_{X,k} \mid \mbox{$A$ is
    surjective and $B$ is injective}\}.
  \]
  The group $\hat G_k$ is defined as $\hat G_k = \GL(I) \times \GL(W)
  \times \GL(J)$.
  An element $(a,b,c)$ of $\hat G_k$ operates on $(A,B) \in \hat
  \sQ_{X,k}$ by $(a,b,c).(A,B)=(a^{-1} A b^\tra,b^{-\tra} B c)$, which still
  lies in $\hat \sQ_{X,k}$. Clearly, $\hat G_k$ also acts on $\hat \sQ_{X,k}^\circ$.

  Now we fix an isomorphism $u : J \to I^*$.
  Note that, given $A \in \sQ_{X,k}^\circ$, we can associate with $A$ the pair
  $(A,D \, A^\tra u)$ which of course lies in $\hat \sQ_{X,k}^\circ$.
  We get a closed embedding $\Phi : \sQ_{X,k}^\circ \mono \hat \sQ_{X,k}^\circ$.
  Next, given a pair $(\xi,\eta) \in G_k$, we associate with
  $(\xi,\eta)$ the element $(\xi,\eta,u^{-1}\xi^{-\tra}u)$ of $\hat G_k$,
  so that $G_k$ is a closed subgroup of $\hat G_k$.
  We observe that the action of $\hat G_k$ on $\hat \sQ_{X,k}^\circ$
  is compatible with that $G_k$ on $\sQ_{X,k}^\circ$ under these
  inclusions.
  Indeed, for all $A\in \sQ_{X,k}^\circ$ and all $(\xi,\eta) \in G_k$
  we have:
  \[
  (\xi,\eta).\Phi(A)=(\xi^{-1}A\eta^\tra,\eta^{-\tra}DA^\tra \xi^{-\tra} u) =
  (\xi^{-1}A\eta^\tra,D \eta A^\tra \xi^{-\tra}u)=\Phi((\xi,\eta).A), 
  \]
  which in turn follows from $\eta^\tra D \eta = D$.

  From the proof of 
  Lemma \ref{quadric-basic}, Proposition \ref{monad-5} and Proposition \ref{v22-basic},
  we see that any element in the moduli space $\MI_X(k)$ comes from a
  pair $(A,B)$ (with in fact $B = D \, A^\tra$), identified up to
  $\hat G_k$-action.
  By the standard argument involving Beilinson-type spectral sequences, 
  we obtain that $\MI_X(k)$ is  isomorphic to the quotient
  $\hat \sQ_{X,k}$ by the action of $\hat G_k$, where the quotient
  map associates with $(A,B)$ the cohomology sheaf $\ker(A)/\im(B)$ of the
  monad obtained by $(A,B)$. 
  The same argument as in \cite{le_potier:construction}
  shows that the points of $\hat \sQ_{X,k}$ are stable for the action
  of $\hat G_k$, and one can prove, just like in \cite[Proposition 24,
  part 2]{le_potier:rang-2} that the action of $\hat G_k$ on $\hat
  \sQ_{X,k}$ is proper.
  This says that the quotient of $\hat \sQ_{X,k}$ by the $\hat
  G_k$-action is geometric.
  Moreover, rephrasing \cite[Pages 233-236]{le_potier:rang-2}, we see that
  $\hat \sQ_{X,k} / \hat G_k$ does represent the moduli functor of
  families of instanton bundles, for the monad construction works well
  in families.

  Finally, we see that this quotient is identified with the quotient of 
  $\sQ_{X,k}$ by the group $G_k$.
  This holds if, for any point $(A,B)$ of $\hat \sQ_{X,k}$, we can find a
  point $A_0$ of $\sQ_{X,k}$ such that $(A,B)$ lies in the $\hat G_k$-orbit
  of $\Phi(A_0)$.
  To check this explicitly, given $(A,B) \in \hat \sQ_{X,k}$,
  note that, from the proof of   Lemma \ref{quadric-basic},
  Proposition \ref{monad-5} and Proposition \ref{v22-basic}, we see that there are
  isomorphisms $w : W^* \to W$ and $v : J \to I$ such that
  $B=w^{-1}A^\tra v$, where $w^\tra=(-1)^{r_X+1} w$.
  Of course $w^{-1}$ is congruent to $D$, so that there is an
  isomorphism $\zeta \in \GL(W)$ such that $w^{-1}=\zeta^\tra D \zeta$.
  Then, setting, $A_0=A\zeta^\tra$, we have:
  \[
  (1,\zeta^\tra,v^{-1}u).(A,B)=(A\zeta^\tra,\zeta^{-\tra}B v^{-1}u)=\Phi(A_0),
  \]
  indeed $\zeta^{-\tra}B v^{-1}u=D\zeta A^\tra u$ follows from $B=w^{-1}A^\tra v$
  and $w^{-1}=\zeta^\tra D \zeta$.
  
\end{proof}

\section*{Acknowledgements}

I am grateful to thank Gianfranco Casnati, Adrian Langer, Francesco Malaspina, Giorgio Ottaviani and
Joan Pons Llopis for valuable comments.
Also, I would like to thank the referee for several helpful
suggestions that allowed to improve the 
first version of the paper.

\bibliographystyle{alpha}
\bibliography{bibliography}

\def\cprime{$'$} \def\cprime{$'$} \def\cprime{$'$} \def\cprime{$'$}
  \def\cprime{$'$} \def\cprime{$'$} \def\cprime{$'$}
\begin{thebibliography}{AHDM78}

\bibitem[AC00]{arrondo-costa}
Enrique Arrondo and Laura Costa.
\newblock Vector bundles on {F}ano 3-folds without intermediate cohomology.
\newblock {\em Comm. Algebra}, 28(8):3899--3911, 2000.

\bibitem[AF06]{enrique-dani:v22}
Enrique Arrondo and Daniele Faenzi.
\newblock Vector bundles with no intermediate cohomology on {F}ano threefolds
  of type {$V\sb {22}$}.
\newblock {\em Pacific J. Math.}, 225(2):201--220, 2006.

\bibitem[AHDM78]{adhm}
Michael~F. Atiyah, Nigel~J. Hitchin, Vladimir~G. Drinfel{\cprime}d, and Yuri~I.
  Manin.
\newblock Construction of instantons.
\newblock {\em Phys. Lett. A}, 65(3):185--187, 1978.

\bibitem[Bei78]{beilinson:derived-and-linear}
Alexander~A. Beilinson.
\newblock Coherent sheaves on {${\bf P}\sp{n}$} and problems in linear algebra.
\newblock {\em Funktsional. Anal. i Prilozhen.}, 12(3):68--69, 1978.

\bibitem[BF08]{brafa1:arxiv}
Maria~Chiara {Brambilla} and Daniele {Faenzi}.
\newblock {Vector bundles on Fano threefolds of genus 7 and Brill-Noether
  loci}.
\newblock {\em ArXiv e-print math.AG/0810.3138}, 2008.

\bibitem[BF11]{brafa2}
Maria~Chiara Brambilla and Daniele Faenzi.
\newblock {Moduli spaces of rank-2 ACM bundles on prime Fano threefolds.}
\newblock {\em Mich. Math. J.}, 60(1):113--148, 2011.

\bibitem[BF12]{brafa3:doi}
Maria~Chiara Brambilla and Daniele Faenzi.
\newblock Rank-two stable sheaves with odd determinant on fano threefolds of
  genus nine.
\newblock {\em Mathematische Zeitschrift}, pages 1--26, 2012.

\bibitem[BO95]{bondal-orlov:semiorthogonal-arxiv}
Alexei~I. {Bondal} and Dmitri~O. {Orlov}.
\newblock {Semiorthogonal decomposition for algebraic varieties}.
\newblock {\em ArXiv eprint alg-geom/9506012}, 1995.

\bibitem[Bon90]{bondal:helices}
Alexei~I. Bondal.
\newblock Helices, representations of quivers and {K}oszul algebras.
\newblock In {\em Helices and vector bundles}, volume 148 of {\em London Math.
  Soc. Lecture Note Ser.}, pages 75--95. Cambridge Univ. Press, Cambridge,
  1990.

\bibitem[BPS80]{banica-putinar-schumacher}
Constantin B{\u{a}}nic{\u{a}}, Mihai Putinar, and Georg Schumacher.
\newblock Variation der globalen {E}xt in {D}eformationen kompakter komplexer
  {R}\"aume.
\newblock {\em Math. Ann.}, 250(2):135--155, 1980.

\bibitem[CF12]{coanda-faenzi:doi}
Iustin Coand{\u{a}} and Daniele Faenzi.
\newblock A refined stable restriction theorem for vector bundles on quadric
  threefolds.
\newblock {\em Annali di Matematica Pura ed Applicata}, pages 1--29, 2012.

\bibitem[CMR09]{costa-miro:monads-instanton}
Laura Costa and Rosa~Maria Mir{\'o}-Roig.
\newblock Monads and instanton bundles on smooth hyperquadrics.
\newblock {\em Math. Nachr.}, 282(2):169--179, 2009.

\bibitem[CO03]{costa-ottaviani:multidimensional}
Laura Costa and Giorgio Ottaviani.
\newblock {Nondegenerate multidimensional matrices and instanton bundles}.
\newblock {\em Trans. Am. Math. Soc.}, 355(1):49--55, 2003.

\bibitem[Dru00]{druel:cubic-3-fold}
St{\'e}phane Druel.
\newblock Espace des modules des faisceaux de rang 2 semi-stables de classes de
  {C}hern {$c\sb 1=0,\ c\sb 2=2$} et {$c\sb 3=0$} sur la cubique de {${\bf
  P}\sp 4$}.
\newblock {\em Internat. Math. Res. Notices}, (19):985--1004, 2000.

\bibitem[ES84]{ein-sols}
Lawrence Ein and Ignacio Sols.
\newblock Stable vector bundles on quadric hypersurfaces.
\newblock {\em Nagoya Math. J.}, 96:11--22, 1984.

\bibitem[Fae05]{dani:v5}
Daniele Faenzi.
\newblock Bundles over the {F}ano threefold {$V\sb 5$}.
\newblock {\em Comm. Algebra}, 33(9):3061--3080, 2005.

\bibitem[Fae07a]{faenzi:v22}
Daniele Faenzi.
\newblock Bundles over {F}ano threefolds of type {$V\sb {22}$}.
\newblock {\em Ann. Mat. Pura Appl. (4)}, 186(1):1--24, 2007.

\bibitem[Fae07b]{faenzi:SL2-instanton}
Daniele Faenzi.
\newblock Homogeneous instanton bundles on {$\Bbb P^3$} for the action of
  {${\rm SL}(2)$}.
\newblock {\em J. Geom. Phys.}, 57(10):2146--2157, 2007.

\bibitem[Gor90]{gorodentsev:exceptional}
Alexei~L. Gorodentsev.
\newblock Exceptional objects and mutations in derived categories.
\newblock In {\em Helices and vector bundles}, volume 148 of {\em London Math.
  Soc. Lecture Note Ser.}, pages 57--73. Cambridge Univ. Press, Cambridge,
  1990.

\bibitem[Har80]{hartshorne:stable-reflexive}
Robin Hartshorne.
\newblock Stable reflexive sheaves.
\newblock {\em Math. Ann.}, 254(2):121--176, 1980.

\bibitem[HL97]{huybrechts-lehn:moduli}
Daniel Huybrechts and Manfred Lehn.
\newblock {\em The geometry of moduli spaces of sheaves}.
\newblock Aspects of Mathematics, E31. Friedr. Vieweg \& Sohn, Braunschweig,
  1997.

\bibitem[Huy06]{huybrechts:fourier-mukai}
Daniel Huybrechts.
\newblock {\em Fourier-{M}ukai transforms in algebraic geometry}.
\newblock Oxford Mathematical Monographs. The Clarendon Press Oxford University
  Press, Oxford, 2006.

\bibitem[IM00a]{iliev-markushevich:degree-14}
Atanas Iliev and Dimitri~G. Markushevich.
\newblock The {A}bel-{J}acobi map for a cubic threefold and periods of {F}ano
  threefolds of degree 14.
\newblock {\em Doc. Math.}, 5:23--47 (electronic), 2000.

\bibitem[IM00b]{iliev-markushevich:degree-quartic}
Atanas Iliev and Dimitri~G. Markushevich.
\newblock Quartic 3-fold: {P}faffians, vector bundles, and half-canonical
  curves.
\newblock {\em Michigan Math. J.}, 47(2):385--394, 2000.

\bibitem[IM07a]{iliev-manivel:genus-8}
Atanas Iliev and Laurent Manivel.
\newblock Pfaffian lines and vector bundles on {F}ano threefolds of genus 8.
\newblock {\em J. Algebraic Geom.}, 16(3):499--530, 2007.

\bibitem[IM07b]{iliev-markushevich:sing-theta:asian}
Atanas Iliev and Dimitri~G. Markushevich.
\newblock Parametrization of sing {$\Theta$} for a {F}ano 3-fold of genus 7 by
  moduli of vector bundles.
\newblock {\em Asian J. Math.}, 11(3):427--458, 2007.

\bibitem[IP99]{fano-encyclo}
Vasilii~A. Iskovskikh and Yuri.~G. Prokhorov.
\newblock Fano varieties.
\newblock In {\em Algebraic geometry, V}, volume~47 of {\em Encyclopaedia Math.
  Sci.}, pages 1--247. Springer, Berlin, 1999.

\bibitem[Isk77]{iskovskih:I}
Vasilii~A. Iskovskih.
\newblock Fano threefolds. {I}.
\newblock {\em Izv. Akad. Nauk SSSR Ser. Mat.}, 41(3):516--562, 717, 1977.

\bibitem[Isk78]{iskovskih:II}
Vasilii~A. Iskovskih.
\newblock Fano threefolds. {II}.
\newblock {\em Izv. Akad. Nauk SSSR Ser. Mat.}, 42(3):506--549, 1978.
\newblock English translation in Math. U.S.S.R. Izvestija \textbf{12} (1978)
  no. 3 , 469--506 (translated by Miles Reid).

\bibitem[JV11]{jardim-verbitsky}
Marcos {Jardim} and Misha {Verbitsky}.
\newblock {Trihyperkahler reduction and instanton bundles on $CP^3$}.
\newblock {\em ArXiv e-print math.AG/1103.4431}, 2011.

\bibitem[Kap88]{kapranov:derived-homogeneous}
Mikhail~M. Kapranov.
\newblock On the derived categories of coherent sheaves on some homogeneous
  spaces.
\newblock {\em Invent. Math.}, 92(3):479--508, 1988.

\bibitem[Kuz96]{kuznetsov:V22}
Alexander~G. Kuznetsov.
\newblock An exceptional set of vector bundles on the varieties {$V\sb {22}$}.
\newblock {\em Vestnik Moskov. Univ. Ser. I Mat. Mekh.}, (3):41--44, 92, 1996.

\bibitem[Kuz97]{kuznetsov:V22-preprint}
Alexander~G. Kuznetsov.
\newblock Fano threefolds {$V\sb{22}$}.
\newblock Preprint MPI 1997--24, 1997.

\bibitem[Kuz12]{kuznetsov:instanton}
Alexander~G. Kuznetsov.
\newblock Instanton bundles on {F}ano threefolds.
\newblock {\em Cent. Eur. J. Math.}, 10(4):1198--1231, 2012.

\bibitem[Lan08]{langer:quadrics}
Adrian Langer.
\newblock {$D$}-affinity and {F}robenius morphism on quadrics.
\newblock {\em Int. Math. Res. Not. IMRN}, (1):Art. ID rnm 145, 26, 2008.

\bibitem[LP79]{le_potier:rang-2}
Joseph Le~Potier.
\newblock Fibr\'es stables de rang {$2$} sur {${\bf P}_{2}({\bf C})$}.
\newblock {\em Math. Ann.}, 241(3):217--256, 1979.

\bibitem[LP94]{le_potier:construction}
Joseph Le~Potier.
\newblock \`{A} propos de la construction de l'espace de modules des faisceaux
  semi-stables sur le plan projectif.
\newblock {\em Bull. Soc. Math. France}, 122(3):363--369, 1994.

\bibitem[Mar76]{maruyama:openness}
Masaki Maruyama.
\newblock Openness of a family of torsion free sheaves.
\newblock {\em J. Math. Kyoto Univ.}, 16(3):627--637, 1976.

\bibitem[Mar80]{maruyama:boundedness-small}
Masaki Maruyama.
\newblock Boundedness of semistable sheaves of small ranks.
\newblock {\em Nagoya Math. J.}, 78:65--94, 1980.

\bibitem[MT01]{markushevich-tikhomirov}
Dimitri~G. Markushevich and Alexander~S. Tikhomirov.
\newblock The {A}bel-{J}acobi map of a moduli component of vector bundles on
  the cubic threefold.
\newblock {\em J. Algebraic Geom.}, 10(1):37--62, 2001.

\bibitem[MT03]{markushevich-tikhomirov:double-solid}
Dimitri~G. Markushevich and Alexander~S. Tikhomirov.
\newblock A parametrization of the theta divisor of the quartic double solid.
\newblock {\em Int. Math. Res. Not.}, (51):2747--2778, 2003.

\bibitem[MT10]{markushevich-tikhomirov:rationality}
Dimitri~G. {Markushevich} and Alexander~S. {Tikhomirov}.
\newblock {Rationality of instanton moduli}.
\newblock {\em ArXiv e-prints}, December 2010.

\bibitem[MU83]{mukai-umemura}
Shigeru Mukai and Hiroshi Umemura.
\newblock Minimal rational threefolds.
\newblock In {\em Algebraic geometry (Tokyo/Kyoto, 1982)}, volume 1016 of {\em
  Lecture Notes in Math.}, pages 490--518. Springer, Berlin, 1983.

\bibitem[Muk92]{mukai:fano-3-folds}
Shigeru Mukai.
\newblock Fano {$3$}-folds.
\newblock In {\em Complex projective geometry (Trieste, 1989/Bergen, 1989)},
  volume 179 of {\em London Math. Soc. Lecture Note Ser.}, pages 255--263.
  Cambridge Univ. Press, Cambridge, 1992.

\bibitem[Orl91]{orlov:V5}
Dmitri~O. Orlov.
\newblock Exceptional set of vector bundles on the variety {$V\sb 5$}.
\newblock {\em Vestnik Moskov. Univ. Ser. I Mat. Mekh.}, (5):69--71, 1991.

\bibitem[OS94]{ottaviani-szurek}
Giorgio Ottaviani and Micha{\l} Szurek.
\newblock On moduli of stable {$2$}-bundles with small {C}hern classes on
  {$Q\sb 3$}.
\newblock {\em Ann. Mat. Pura Appl. (4)}, 167:191--241, 1994.
\newblock With an appendix by Nicolae Manolache.

\bibitem[Ott88]{ottaviani:spinor}
Giorgio Ottaviani.
\newblock Spinor bundles on quadrics.
\newblock {\em Trans. Amer. Math. Soc.}, 307(1):301--316, 1988.

\bibitem[Pro90]{prokhorov:exotic}
Yuri~G. Prokhorov.
\newblock Exotic {F}ano varieties.
\newblock {\em Vestnik Moskov. Univ. Ser. I Mat. Mekh.}, (3):34--37, 111, 1990.

\bibitem[Sch01]{schreyer:V22}
Frank-Olaf Schreyer.
\newblock Geometry and algebra of prime {F}ano 3-folds of genus 12.
\newblock {\em Compositio Math.}, 127(3):297--319, 2001.

\bibitem[Tak89]{takeuchi:birational}
Kiyohiko Takeuchi.
\newblock Some birational maps of {F}ano {$3$}-folds.
\newblock {\em Compositio Math.}, 71(3):265--283, 1989.

\bibitem[{Tik}11]{tikhomirov:odd}
Alexander~S. {Tikhomirov}.
\newblock {Moduli of mathematical instanton vector bundles with odd $c_2$ on
  projective space}.
\newblock {\em ArXiv e-print math.AG/1101.3016}, 2011.

\bibitem[Wal78]{wall:nets-of-quadrics}
C.~T.~C. Wall.
\newblock Nets of quadrics, and theta-characteristics of singular curves.
\newblock {\em Philos. Trans. Roy. Soc. London Ser. A}, 289(1357):229--269,
  1978.

\end{thebibliography}

\end{document}